\newcommand\datum{December 3, 2022}
\documentclass[reqno]{amsart}
\usepackage{amsthm, amscd, amsfonts,graphicx}
\usepackage{color}
\usepackage{enumerate}
\usepackage{verbatim}
\numberwithin{equation}{section}
\newenvironment{enumeratei}{\begin{enumerate}[\quad\upshape (i)]} {\end{enumerate}}
\theoremstyle{plain}
 \newtheorem{theorem}{Theorem}[section]
 \newtheorem{lemma}[theorem]{Lemma}
 \newtheorem{proposition}[theorem]{Proposition}
 \newtheorem{observation}[theorem]{Observation}

\theoremstyle{definition}
 
 \newtheorem{remark}[theorem]{Remark}
 
 \newtheorem{example}[theorem]{Example}

\theoremstyle{remark}

 \newtheorem*{aremark}{Remark}
\DeclareMathOperator \thrSi{\mathbf{Si_3}}
\newcommand \spleq {\leq_{\textup{sp}}} 
\newcommand \vonal {\noalign{\hrule}}
\newcommand \mvar [1] {\var #1^{{{\scriptscriptstyle\pmb{{\textup{\text{\raisebox{1.5pt} m}}}}}}} }
\newcommand \varmplu{\var M^+}
\newcommand \alllat{\pmb\Lambda}
\newcommand \bhhrom{\pmb{\mathbf B}_{23}}
\newcommand \Thhrom {\var T_{23}}
\newcommand \Ker[1]{\textup{Ker}(#1)}

\DeclareMathOperator \HSP {\mathbf{HSP}}

\DeclareMathOperator \Hs{\mathbf{H}}
\DeclareMathOperator \Ss{\mathbf{S}}
\DeclareMathOperator \Ps{\mathbf{P}}
\DeclareMathOperator \HSs{\mathbf{HS}}
\DeclareMathOperator \Si{\mathbf{Si}}
\DeclareMathOperator \Aut {\textup{Aut}}

\newcommand \block[2] {#1/#2}

\newcommand \AS{\textup{AS}}
\newcommand \CS{\textup{CS}}
\newcommand \DS{\textup{DS}}

\newcommand \vx {\vec x}
\newcommand \vy {\vec y}
\newcommand \vz {\vec z}

\newcommand \var [1] {\mathcal{#1}}
\newcommand \notdiv {\mathrel{\not{\kern -0.27pt|}} }

\newcommand \Nnul {\mathbb N_0}
\newcommand \vfree[2] {\textup{FL}_{#1}(#2)}

\newcommand \ffree[1] {\textup{FL}(3)}

\newcommand \afree[2] {\textup{FL}_{#1}(#2)}

\newcommand \Ats[1] {\textup {At}({#1})}
\newcommand \Coats[1] {\textup {Ct}({#1})}

\newcommand \muk {\mu_K}

\newcommand \dual [1]  {{#1}{}^{\text{dual}}}

\newcommand \filter{\mathord{\uparrow}}
\newcommand \lideal[1]{\mathord{\downarrow}_{\kern-1pt#1}}
\newcommand \lfilter[1]{\mathord{\uparrow}_{\kern-1pt#1}}
\newcommand \restrict[2] {{#1\rceil_{#2}}}

\newcommand \tbf [1] {\textbf{#1}} 
 
\newcommand \set[1] {\{#1\}}
\renewcommand \phi {\varphi}
\newcommand \chain[1]{\mathsf C_{#1}}

\newcommand \red [1] {{\color{red}#1\color{black}}}

\newcommand \nothing [1] {}
\newcommand \magenta [1] {{\color{magenta}#1\color{black}}}

\begin{document}
\title
{Atoms and coatoms in three-generated lattices}

\author[G.\ Cz\'edli]{G\'abor Cz\'edli}
\address{University of Szeged, Bolyai Institute, Szeged,
Aradi v\'ertan\'uk tere 1, Hungary 6720}
\email{czedli@math.u-szeged.hu}
\urladdr{http://www.math.u-szeged.hu/~czedli/}

\thanks{This research was supported by the National Research, Development and Innovation Fund of Hungary under funding scheme K 134851.}

\dedicatory{Dedicated to the memory of my parents, M\'aria and Gy\"orgy}

\date{\hfill {\tiny{\magenta{(\tbf{Always} check the author's website for possible updates!) }}}\  \red{\datum}}


\begin{abstract} In addition to the  unique cover $\var M^+$ of the variety of modular lattices, we also deal with those twenty-three \emph{known} covers of  $\var M^+$ that can be extracted from the literature.  
For $\var M^+$ and for each of these twenty-three known varieties  covering  it, we determine what the pair formed by the number of atoms and that of coatoms of a three-generated lattice belonging to the variety in question can be. Furthermore, for each variety $\var W$ of lattices that is obtained by forming the join of some of the  twenty-three varieties mentioned above, that is, for $2^{23}$ possible choices of $\var W$,  we 
determine how many atoms a three-generated lattice belonging to $\var W$ can have.  The greatest number of atoms occurring in this way is only six. In order to point out that this need not be so for larger varieties, we construct a $47\,092$-element three-generated lattice that has exactly eighteen atoms.
In addition to purely lattice theoretical proofs, which constitute the majority of the paper,  some computer-assisted arguments are also presented. 
\end{abstract}

%

\maketitle
\color{black}

\section{Introduction and target}
This paper is devoted to the question that, for some varieties $\var V$ of lattices, 
 how many atoms and how many coatoms a three-generated lattice in $\var V$ can have.

\subsection{Outline}\label{subsect:outline} The paper is structured as follows. Subsections~\ref{subsect:notation} (the next subsection) gives the basic concept and notation used in the paper. Subsection~\ref{subsect:earlier} recalls all the results that have previously been known on the number of atoms in three-generated lattices; see statements \eqref{eqknXmpHla}--\eqref{eqknXmpHld}. 
Subsection~\ref{subsect:goal}, after introducing some further notation, formulates our goal; note that the  main result, Theorem~\ref{thmmain}, comes later. Section~\ref{sectionkeylemma} proves some lemmas. The (Key) Lemma~\ref{keylemma} of this section is worth separate mentioning since it could be useful in extending our results to more lattice varieties. 
Section~\ref{sect:intvl} contributes a little to our knowledge of the lattice of all lattice varieties; in particular,  Proposition~\ref{prophszHfd}\eqref{prophszHfdd} asserts that the varieties occurring in the Main Theorem  form a $2^{23}$-element boolean  interval in this lattice.
Section~\ref{sect:somespectra} determines the possible numbers of atoms and, in some cases, these numbers jointly with the numbers of coatoms for \emph{some} of the $2^{23}$ lattice varieties described in the previous section. 
Section~\ref{sectionmorabout} determines these possible numbers of atoms for  \emph{each} of the  $2^{23}$ lattice varieties and formulates the main result of the paper, Theorem~\ref{thmmain}. Finally, Section~\ref{sect:howfar} contains some additional observations on the numbers of atoms. In particular, 
Example~\ref{exmpl:18} presents a three-generated lattice with eighteen atoms; this lattice consists of $47\,092$ elements. 
Note that, as opposed to the earlier sections, Sections~\ref{sectionmorabout} and \ref{sect:howfar} include some computer-assisted arguments in addition to theoretical considerations.

\subsection{Basic notation}\label{subsect:notation} 
For an at most countable lattice $L$, let $\Ats L$ and $\Coats L$ stand for the set of atoms of $L$ and that of coatoms of $L$, respectively. The acronyms come from \emph{At}oms and \emph{C}oa\emph{t}oms. The cardinality $|\Ats L|$ is in $\Nnul:=\set{0,1,2,3,\dots}$ or it is $\aleph_0$, and the same holds for $\Coats L$.  
For a variety $\var V$ of lattices, we define three sorts of \emph{spectra} of $\var V$ as follows.
\begin{align}
\AS(\var V)&=\{|\Ats L|: L\in\var V\textup{ and }L\text{ is three-generated} \},
\label{eqlHzfQa}\\
\CS(\var V)&=\{|\Coats L|: L\in\var V\textup{ and }L\text{ is three-generated} \},\label{eqlHzfQb}\\
\DS(\var V)&=\{(|\Ats L|,|\Coats L|): L\in\var V,\textup{ }L\text{ is three-generated} \}.\label{eqlHzfQc}
\end{align}
These spectra are called the \emph{Atom Spectrum}, the \emph{Coatom Spectrum}, and the \emph{Double Spectrum} of $\var V$, respectively; the capital letters here are to explain the acronyms.

\subsection{Earlier results on the numbers of atoms}\label{subsect:earlier}
To present some examples for the concepts introduced in \eqref{eqlHzfQa}--\eqref{eqlHzfQc}, 
\begin{equation}\left.
\parbox{9cm}{let 
$\var M$, $\var D$, and $\var L$ be the variety of modular lattices, that of distributive lattices, and that of all lattices, respectively.}\,\,
\right\}
\label{pbxMDLjl}
\end{equation}
By  Cz\'edli~\cite{czgnumatoms} and duality, we know that 
\begin{align}
&\AS(\var M)=\CS(\var M)=\set{1,2,3}, \text{ so } \AS(\var D)=\CS(\var D)=\set{1,2,3}, \label{eqknXmpHla} \\
&\set{0,1,2,3,4}\subseteq  \AS(\var L) \cap \CS(\var L)\text{ but }(0,0)\notin \DS(\var L).\text{ Trivially,}\label{eqknXmpHlb}\\
&\text{if }(1,k)\text{ or }(k,1)\text{ is in }\DS(\var L),\text{ then }k\in\set{1,2}. \text{ Also,}\label{eqknXmpHlc}\\
&\DS(\var D)=\DS(\var M)=
\{(1,1),(1,2), (2,1), (2,2), (2,3), (3,2), (3,3)\}.\label{eqknXmpHld}
\end{align}
Note that \eqref{eqknXmpHld} follows from \eqref{eqknXmpHla} since each pair listed in \eqref{eqknXmpHld} is easy to represent; for example, $(2,2)\in \DS(\var D)$ and $(2,3)\in \DS(\var D)$ are  witnessed by the lattices labeled by $(2,2)$ and $(2,3)$ in  Figure~\ref{figaa}. In these two lattices, the generators are black-filled.
Since for arbitrary varieties $\var W_1$, $\var W_2$, and $\var W_3$ of lattices,
\begin{equation}\left.
\parbox{9cm}{if we have that $\var W_1\subseteq \var W_2\subseteq \var W_3$ and 
$\AS(\var W_1)=\AS(\var W_3)$, then $\AS(\var W_2)=\AS(\var W_3)$, and analogously with $\CS$ and $\DS$,}\,\,\right\}
\label{eqknXmpHle}
\end{equation}
we obtain $\AS(\var W)$,  $\CS(\var W)$,  $\DS(\var W)$ from \eqref{eqknXmpHla} and \eqref{eqknXmpHld} for every lattice variety $\var W$ between $\var D$ and $\var M$.  Note that there are continuously many such varieties $\var W$; see, for example, Hutchinson and Cz\'edli~\cite{hutchczg}.

\begin{figure}[ht]
\centerline
{\includegraphics[scale=1.0]{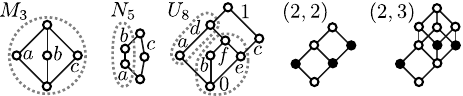}}
\caption{$M_3$, $N_5$, $U_8$, and representing two $(|\Ats L|,|\Coats L|)$ pairs}
\label{figaa}
\end{figure}

Examples \eqref{eqknXmpHla}--\eqref{eqknXmpHle} represent what has previously been known about the three spectra we have defined. However,  there are continuously many lattice varieties \emph{not} included in $\var M$ and so not belonging to the scope of  \eqref{eqknXmpHla}--\eqref{eqknXmpHle}. Hence, the examples above also show  how little has been known about the number of atoms and that of coatoms in a three-generated lattice in general.

\begin{figure}[ht]
\centerline
{\includegraphics[scale=1.0]{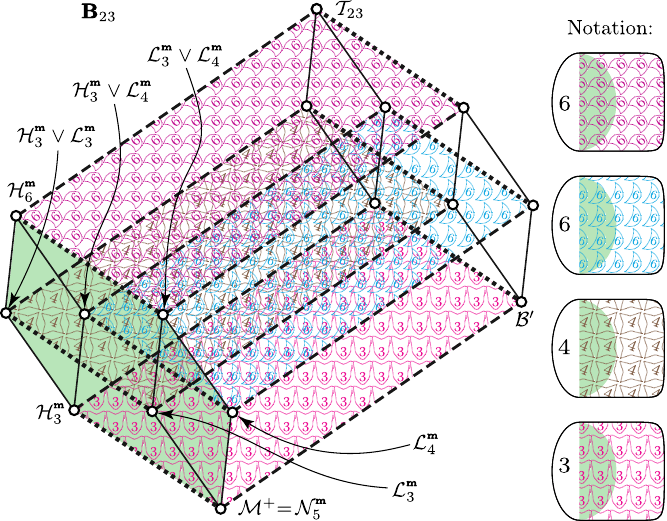}}
\caption{A 16-element $\set{0,1}$-sublattice of $\bhhrom$; the solid lines are coverings in $\bhhrom$ while the dotted lines and the dashed lines are  8-element and $2^{19}$-element (boolean) intervals in $\bhhrom$. $\bhhrom$ is partitioned into four \emph{layers}; their fill patterns are explained on the right.}
\label{figdd}
\end{figure}

\subsection{Our goal}\label{subsect:goal}
Our goal in this paper is to enrich the above-mentioned little knowledge by proving some facts about the spectra of some varieties that are slightly larger than the variety $\var M$ of modular lattices. In the lattice $\alllat$ of all lattice varieties,  $\var M$  has a unique cover $\mvar N_5=\varmplu$; it is the least variety containing $N_5$ in Figure~\ref{figaa} and all modular lattices.  (Here $\varmplu$ is the traditional notation but we also write  $\mvar N_5$, which fits better in the present paper.) Based on deep classical results, mainly Bjarni J\'onsson's results, it needs only a trivial consideration to present twenty-three lattice varieties covering $\varmplu$. These twenty-three varieties will be called the \emph{known covers} of $\varmplu$ since  $\varmplu$ may have further covers. We are going to point out in Proposition~\ref{prophszHfd}\eqref{prophszHfdd} that
\begin{equation}\left.
\parbox{9cm}{in $\alllat$, the lattice of all lattice varieties, the above-mentioned known covers generate a boolean sublattice of length 23 (and of size $2^{23}$); we denote this sublattice by $\bhhrom$.}\,\,\right\}
\label{pbxbshrBhrm}
\end{equation} 
In fact, Proposition~\ref{prophszHfd}\eqref{prophszHfdd} we will say more, namely, 
\begin{equation}
\text{$\bhhrom$ is an interval in $\alllat$.}
\label{eqtxtBnrVl}
\end{equation} 
The bottom of $\bhhrom$ is  $\varmplu=\mvar N_5$ and its atoms are the twenty-three known covers of $\varmplu$. 
\begin{equation}\left.
\parbox{9cm}{We denote the top of $\bhhrom$ by $\Thhrom$; this lattice variety is the join of the twenty-three known covers of $\varmplu=\mvar N_5$.}\,\,\right\}
\label{pbxthTlvRhrM}
\end{equation}
Although a  $2^{23}$-element lattice cannot be drawn in the practice, 
the schematic diagram given in Figure~\ref{figdd} gives some insight into it; the details will be explained in Section~\ref{sectionmorabout}.
For each of the $2^{23}$ lattice varieties belonging to $\bhhrom$, we \emph{determine the atom spectrum} of the variety in question. The description of atom spectra of members of $\bhhrom$ is even visualized by Figure~\ref{figdd}; we will later explain how.  
Since this description can be dualized in a trivial way, we are not going to  pay separate attention to coatom spectra.  
The \emph{double spectra} create so much computational difficulty that they are determined only for a quarter of the varieties belonging to  $\bhhrom$, including the twenty-three known covers of $\mvar N_5$.

It turns out that the largest number in $\Ats {\Thhrom}$ is 6, so a three-generated lattice in a variety belonging to $\bhhrom$ has at most six atoms. Even six is larger than all what previously have been known, but it is not the largest number of atoms of a three-generated lattice in this paper. Let $U_8$ denote the eight element lattice given by Figure~\ref{figaa}, and let $\var U_8:=\HSP\set{U_8}$ be the variety generated by $U_8$. As it will be pointed out, $\mvar U_8:=\var U_8\vee \var M$ is not in $\bhhrom$ but it covers one of the members of $\bhhrom$ in $\alllat$. Witnessed by a $47\,092$-element three-generated lattice belonging to $\var U_8$, we show that $18\in\AS(\var U_8)\subseteq \AS(\mvar U_8)$.

\begin{aremark}
With the exception of $\var L$, see \eqref{pbxMDLjl}, the free lattice $\vfree{\var W}3$  of $\var W$ on three generators is finite by trivial reasons in each of the lattice varieties $\var W$ occurring in the present paper.  Hence, up to isomorphism, there are only finitely many three-generated lattices in these varieties, whereby each of the three spectra is a finite set for these varieties.  Apart from straightforward consequences of the results we are going to prove here, we do not know anything about the spectra of varieties $\var W$ with $\vfree{\var W}3$ infinite.
\end{aremark}

\section{Some lemmas}\label{sectionkeylemma}
The first lemma we are going to formulate belongs to the folklore.
Its particular case for free algebras and automorphisms is mentioned in page 272 of Berman and Wolk~\cite{bermanwolk}, and a more general case with homomorphisms can also be extracted from  \cite[page 273]{bermanwolk}. For later reference and for the reader's convenience, we are going to give an explicit formulation and a short proof. Before stating the lemma, we need some preparation. Although we are only interested in lattices in the present paper, we can allow more general algebras in the first lemma without extra work. 

The least congruence of an algebra $K$ is called the \emph{zero congruence} of $K$; it is denoted by $\Delta$ or, 
if $K$ needs to be specified, by $\Delta_K$. An algebra $K$ is \emph{subdirectly irreducible} if it has a least nonzero congruence; this congruence is called the \emph{monolith} of $K$ and it is denoted by $\mu=\mu_K$. We use the notation $L\spleq \prod_{i\in I}L_i$ to denote that the $L_i$, $i\in I$, are algebras and $L$ is a subdirect product of them.  That is, $L$ is a subalgebra of the direct product $\prod_{i\in I}L_i$ such that the \emph{projection map}
\begin{equation}
\pi_i: L\to L_i, \text{ defined by } u\mapsto u(i),
\label{eqZhhsPnkZrZtnk}
\end{equation}
is surjective for every $i\in I$. Let $X$ be a generating set of $L$. We say that 
\begin{equation}\left.
\parbox{9cm}{a homomorphism $\psi\colon L_i\to L_j$  \emph{criticizes the generating set $X\,$} if $i,j\in I$, $i\neq j$, and  $\psi(x(i))=x(j)$ for all $x\in X$. If no homomorphism criticizes $X$, then $L\spleq \prod_{i\in I}L_i$ is an \emph{irredundant subdirect product} (with respect to $X$).}\,\,\right\}
\label{pbxRsPchtmrGrthH}
\end{equation}
Algebras consisting of at least two elements are said to be \emph{nontrivial}. For $|I|\geq 2$, if $L\spleq \prod_{i\in I}L_i$ above is an irredundant subdirect product, then all the $L_i$, $i\in I$, are nontrivial. 
If the condition given in \eqref{pbxRsPchtmrGrthH} fails, then the subdirect product is \emph{redundant}  (with respect to $X$). 

\begin{lemma}\label{lemmacZTkglnTxp} Let $L$ be a nontrivial finite algebra with a fixed generating set $X$. 
Then, up to isomorphism, $L$ is an \emph{irredundant} subdirect product  $L\spleq \prod_{i\in I}L_i$ with respect to $X$  in the sense of \eqref{pbxRsPchtmrGrthH} such that $I$ is a \emph{finite} index set and, for every $i\in I$,  $L_i$ is a finite subdirectly irreducible algebra generated by $\set{\pi_i(x):x\in X}$.
\end{lemma}

\begin{proof}
By a classical theorem of G. Birkhoff, see \cite[Theorem 1]{birkhoff}, $L$ is a subdirect product of finitely many subdirectly irreducible algebras. 
Hence, we can choose $L\spleq \prod_{i\in I}L_i$ such that the finite number of factors $|I|$  is minimal. 
Since a surjective homomorphism takes a generating set to a generating set, 
$L_i$ is  generated by $\set{\pi_i(x):x\in X}$ for all $i$. Also, $|L_i|=|\pi_i(L)|\leq |L|$ shows that $L_i$ is finite. 
We claim that our subdirect product is irredundant (with respect to $X$). Suppose the contrary, and pick $j,k\in I$ and a homomorphism $\psi \colon L_j\to L_k$ that criticizes $X$. Let $J:=I\setminus \set{k}$. Note that $j\in J$ since $j\neq k$.  Let us agree that 
\begin{equation}\left.
\parbox{6cm}{the restriction of a map $\kappa$ to a subset $A$ of its domain will be denoted by $\restrict\kappa A$.}
\,\,\right\}
\label{pbxthRsxRkNjlS}
\end{equation}
For $u\in \prod_{i\in I}L_i$, we let $u':=\restrict u J\in \prod_{i\in J}L_i$. Also, let $L':=\set{u': u\in L}$ and $X':=\set{x': x\in X}$. 
Clearly,  $L'\spleq \prod_{i\in J}L_i$, and  the map $\phi\colon L\to L'$, defined by $u\mapsto u'$, is a surjective homomorphism. Since $\phi(X)=X'$, it follows that $X'$ generates $L'$. Next, let $u\in L$ and pick a term $t$ and elements $x_1,\dots, x_s\in X$ such that $u=t(x_1,\dots,x_s)$ holds in $L$. Since $t$ commutes with $\phi$, we have that  $u'=t(x'_1,\dots,x'_s)$. We obtain that
\begin{align*} 
u(k)&=t(x_1,\dots,x_s)(k)=t(x_1(k),\dots, x_s(k)) \cr
&= t(\psi(x_1(j)),\dots,\psi(x_s(j)))= \psi(t(x_1(j),\dots,x_s(j)))\cr
&=  \psi(t(x'_1(j),\dots,x'_s(j))) =\psi(t(x'_1,\dots,x'_s)(j))=\psi(u'(j)).
\end{align*}
Since we also have, trivially, that  $u(i)=u'(i)$ for $i\in I\setminus\set k$, it follows that $u'=\phi(u)$ determines $u$, whereby $\phi$ is injective. So $\phi$ is an isomorphism and we can identify $X'=\phi(X)$ with $X$. Hence, up to isomorphism, $L\spleq \prod_{i\in I}L_i$ and $X$ can be replaced by $L'\spleq \prod_{i\in J}L_i$ and $X'$. This is a contradiction since  $|J|<|I|$ but $|I|$ was assumed to be minimal. Therefore, the subdirect product $L\spleq \prod_{i\in I}L_i$ is irredundant, as required.
\end{proof}

As usual, for a class $\var X$ of lattices, the class of homomorphic images, that of sublattices, and that of direct products of lattices belonging to $\var X$ will be denoted by $\Hs\var X$, $\Ss\var X$, and $\Ps\var X$, respectively. 
By the classical ``HSP theorem'' of Birkhoff~\cite{birkhoffHSP}, $\HSP{\var X}$, called the \emph{variety generated by $X$}, is the least equationally defined class of lattices that includes $\var X$. For a class $\var X$ of lattices, the class of subdirectly irreducible lattices of $\var X$ will be denoted by $\Si \var X $.
Since we will repeatedly use some celebrated results of J\'onsson~\cite{jonsson}, we formulate them for later references and for the reader's convenience. 
Namely, a particular case of J\'onsson~\cite[Lemma 4.1]{jonsson} asserts that
\begin{equation}\left.
\parbox{6.2cm}{If $\var W_1$ and $\var W_2$ are lattice varieties, then 
$\Si(\var W_1\vee \var W_2)=(\Si\var W_1)\cup(\Si\var W_2)$.
}\,\,\right\}
\label{pbxBJg}
\end{equation}
Also,  J\'onsson~\cite[Corollary 3.4]{jonsson} applied to lattices gives that
\begin{equation}\left.
\parbox{5.2cm}{if $\var X$ is a finite set of finite lattices, then
$\Si(\HSP\var X)\subseteq \HSs\var X$.
}\,\,\right\}
\label{pbxBJktt}
\end{equation}

\begin{lemma}\label{lemmakdlBRbFnW} Let $\var V$ be a variety of lattices, and let $k$ be a positive integer.  Also, let $K$ be a finite lattice. Denote by $\var W$ the variety $\HSP(\var V\cup\set K)$.  If the free lattice $\vfree{\var V}k$ in $\var V$ on $k$ generators is finite, then  every $k$-generated lattice in $\var W$ is finite.
\end{lemma}

\begin{proof} Since $\vfree{\var V}k$ is finite, it has only finitely many quotient lattices. Let $E$ be the direct product of $K$ and these quotient lattices, and note that $E$ is a finite lattice. Let $\var E:=\HSP\set E$. Since $K\in\Hs\set E$, we have that $\HSP\set K\subseteq \var E$. 
We know from Hobby and McKenzie~\cite[Theorem 0.1]{hobbymckenzie} that a finitely generated algebra in a variety generated by a finite set of finite algebras is necessarily finite. Hence, every $k$-generated lattice in $\var E$ is finite. Thus, it suffices to show that every $k$-generated lattice of $\var W$ belongs to $\var E$.  To do so, let $L\in \var W$ be a $k$-generated lattice. 
By  Birkhoff~\cite[Theorem 2]{birkhoff}, $L$ is a subdirect product of  subdirectly irreducible lattices $L_i$, $i\in I$, where $I$ is a (not necessarily finite) index set. As a homomorphic image of $L$, the lattice $L_i$ is generated by at most $k$ elements for every $i\in I$.
Since $L_i\in\Si(\var W)$, \eqref{pbxBJg} gives that $L_i\in\Si(\var V)\cup \Si(\HSP\set K)\subseteq \var V\cup \HSP \set K$, which permits only two cases for an  $i\in I$. First, assume that $L_i$ in $\var V$. Then it is a homomorphic image of $\vfree{\var V}k$. So it is isomorphic to one of the direct factors of $E$. Hence $L_i$ is a homomorphic image of $E$, whereby it belongs to the variety $\HSP\set E=\var E$. 
Second, assume that $L_i\in \HSP\set K$. Then $L_i\in\var E$ since $\HSP\set K\subseteq \var E$. 
We have seen that $L_i\in \var E$ for all $i\in I$. This yields that $L\in \var E$, completing the proof of the lemma.
\end{proof}

To formulate our key lemma, we need to introduce some further concepts. We say that
\begin{equation}\left.
\parbox{8.8cm}{a lattice $L$ satisfies \emph{meet condition \eqref{pbxmeetcond}}
if for each $(u_1,u_2,u_3)\in L^3$ such that  $\set{u_1,u_2,u_3}$ generates $L$, there are at least two pairs $(i,j)$ in $\set{(1,2),(1,3),(2,3)}$ such that 
$u_i\wedge u_j\neq 0$.}
\,\,\right\}
\label{pbxmeetcond}
\end{equation}
Dually, we say that 
\begin{equation}\left.
\parbox{8.8cm}{a lattice $L$ satisfies \emph{join condition \eqref{pbxjoincond}}
if for each $(u_1,u_2,u_3)\in L^3$ such that  $\set{u_1,u_2,u_3}$ generates $L$, 
there are at least two pairs $(i,j)$ in $\set{(1,2),(1,3),(2,3)}$ such that 
$u_i\vee u_j\neq 1$.}
\,\,\right\}
\label{pbxjoincond}
\end{equation}
Clearly, 
\begin{equation}
\parbox{6.7cm}{if $|L|\geq 3$ and $L$ has no three-element generating set, then $L$ satisfies both  meet condition \eqref{pbxmeetcond} and  join condition \eqref{pbxjoincond}.}
\label{pbxcbThcnDzWrtmN}
\end{equation}
As usual, we will write $0\in L$ and $1\in L$ to express that $L$ has a smallest element and a largest element, respectively. Note that, as a trivial consequence of \eqref{pbxcbThcnDzWrtmN},  $0\notin L$ implies the validity of \eqref{pbxmeetcond} and dually.

A congruence $\Theta$ of a lattice $L$ will be called \emph{$0$-separating}
if $0\in L$ and the $\Theta$-block of $0$, denoted by $\block 0\Theta$, is the singleton set $\set0$. We define \emph{$1$-separating} congruences dually. 
A variety of lattices is \emph{nontrivial} if it contains a nonsingleton lattice.
For a finite lattice $K$, a homomorphic image or a sublattice of $K$ is \emph{proper} if it has fewer elements than $K$. 
Now we are ready to formulate the following lemma.

\begin{lemma}[Key Lemma]\label{keylemma} Let $\var V$ be a nontrivial variety of lattices  such that 
the three-generated $\var V$-free lattice $\vfree{\var V}3$ is finite, and  let $K$ be a finite subdirectly irreducible lattice such that all proper homomorphic images and all proper sublattices of $K$ belong to $\var V$. If $\var W=\HSP({\var V\cup\set{K}})$ denotes the lattice variety generated by  $\var V\cup\set{K}$, then  $\vfree{\var W}3$ is also finite and, furthermore, the following three assertions hold.
\begin{enumeratei}
\item\label{keylemmaa} If $K$ satisfies  meet condition \eqref{pbxmeetcond} or the monolith $\muk $ is $0$-separating, then $\AS(\var W)=\AS(\var V)$.
\item\label{keylemmab} If $K$ satisfies  join condition \eqref{pbxjoincond} or the monolith $\muk $ is $1$-separating, then $\CS(\var W)=\CS(\var V)$.
\item\label{keylemmac} Assume that  $K$ satisfies both meet condition \eqref{pbxmeetcond} and  join condition \eqref{pbxjoincond}, or that $\muk $ is both $0$-separating and $1$-separating. Then $\DS(\var W)=\DS(\var V)$.
\end{enumeratei}
\end{lemma}

It will be clear from the proof that a weaker assumption would be sufficient for  the finites of $\vfree{\var W}3$, but we do not need this fact. Note also that if $K\in \var V$, then $\var W=\var V$ and the statement of the lemma trivially holds.

\begin{proof} 
We say that a surjective lattice homomorphism  is \emph{$0$-separating} if so is its congruence kernel. Equivalently, a surjective lattice homomorphism  is 0-separating if it sends nonzero elements to nonzero elements. 
We claim that for arbitrary finite lattices $T_1$ and $T_2$  and a homomorphism $\kappa\colon T_1\to T_2$,
\begin{equation}\left.
\parbox{8cm}{if $\kappa$ is surjective, then
$\kappa(\Ats{T_1})\subseteq \set 0\cup \Ats{T_2}$. If, in addition to its surjectivity,
$\kappa$ is 0-separating, then $\restrict{\kappa}{\Ats{T_1}}\colon
\Ats{T_1}\to  \Ats{T_2}$ is a bijective map.
}\,\,\right\}
\label{pbxkenkVsPkhTjQwBh} 
\end{equation}
For the sake of contradiction, suppose that $p\in\Ats{T_1}$ but $p':=\kappa(p)\notin \set 0\cup \Ats{T_2}$.  Pick an element $q'\in T_2$ such that $0<q'<p'$. Since $\kappa$ is surjective, there is a $q\in T_1$ with $\kappa(q)=q'$. Let $r:=p\wedge q$. Since $r\leq p$ and $\kappa(r)=\kappa(p)\wedge\kappa(q)=p'\wedge q'=q'\neq p'=\kappa(p)$, we have that $r<p$. From $\kappa(0)=0\neq q'=\kappa(r)$ we obtain that $r\neq 0$. Hence, we have obtained $0<r<p$, contradicting $p\in\Ats{T_1}$ and proving the inclusion
$\kappa(\Ats{T_1})\subseteq \set 0\cup \Ats{T_2}$. Now, to prove the second half of  \eqref{pbxkenkVsPkhTjQwBh}, assume that $p\in \Ats{T_1}$.
If we had that $\kappa(p)=0=\kappa(0)$, then $(p,0)$ would belong to the kernel $\Ker\kappa$ of $\kappa$, which is impossible since $\Ker\kappa$  is 0-separating. Hence $\kappa(p)\neq 0$, and it follows from the already proven first half of  \eqref{pbxkenkVsPkhTjQwBh} that  $\restrict{\kappa}{\Ats{T_1}}$ is an $\Ats{T_1}\to  \Ats{T_2}$ map. It is clearly injective since otherwise $\kappa(p)=\kappa(q)$ would hold with some distinct $p,q\in\Ats {T_1}$ and $\kappa(p)=\kappa(p)\wedge \kappa(p)=\kappa(p)\wedge \kappa(q)=\kappa(p\wedge q)=\kappa(0)=0$ would contradict the 0-separability of $\kappa$. To show the surjectivity of $\kappa$, let $p'\in\Ats{T_2}$. Let $p:=\bigwedge\set{q\in T_1: \kappa(q)=p'}$; this is a nonempty and existing meet since  $\kappa$ is surjective and $T_1$ is finite. Note that $p$ is the least preimage of $p'$ since $\kappa(p)=\bigwedge\set{\kappa(q)\in T_1: \kappa(q)=p'}=p'$. Since $p'$ is distinct from $0$, so is $p$. For the sake of contradiction, suppose that $p\notin \Ats{T_1}$ and pick an element $s\in T_1$ with $0<s<p$. Then $\kappa(s)\leq \kappa(p)=p'$. Actually, $\kappa(s)<p'$ since $p$ is the least preimage of $p'$. Also, $\kappa(s)\neq 0$ since $\kappa$ is 0-separating. 
Hence, $0<\kappa(s)<p'$ contradicts that $p'\in\Ats{T_2}$ and proves the surjectivity of $\kappa$. Thus, \eqref{pbxkenkVsPkhTjQwBh} has been proved.

Next, let $L\in \var W$ be a three-generated lattice. By a classical theorem of Birkhoff~\cite{birkhoff}, we can assume that it is a subdirect product 
\begin{equation}L\spleq \prod_{i\in I}L_i,\text{ where $L_i$ is 3-generated and subdirectly irreducible}
\label{eqHzrNmrCtJlRg}
\end{equation}
for all $i\in I$. (At present, we do not claim that $I$ is finite.) By \eqref{pbxBJg}, 
$L_i\in \var V \cup \HSP\set K$ for every $i\in I$. For a moment, let us focus on the possibility that $L_i\in \HSP\set K$. Since $L_i$ is subdirectly irreducible and $K$ is finite, \eqref{pbxBJktt} yields that $L_i\in \HSs\set K$, and there are only two cases. Either  $|L_i|=|K|$ and then  $L_i$ is isomorphic to $K$ and so we can assume that $L_i=K$ in this case, or $|L_i|<|K|$ and then $L_i$ in $\var V$ follows from the assumption on proper homomorphic images and sublattices of $K$. 
If $L_i\notin \HSP\set K$, then $L_i\in \var V \cup \HSP\set K$ gives again that $L_i$ in $\var V$. Hence, letting 
\begin{equation}\left.
\parbox{8.4cm}
{$H:=\set{i\in I: L_i=K}$ and  $J:=\set{i\in I: L_i\not\cong  K\text{ and}
\\ L_i\in\var V}$, we can assume that $I=H\cup J$ and $H\cap J=\emptyset$.}
\,\,\right\}
\label{pbxZhRnWsnpXq} 
\end{equation}
Note that one of $H$ or $J$ can be empty but this will not cause any problem since the direct product of an empty family of lattices is meaningful: it is the trivial lattice, that is, the singleton lattice. 
Clearly, the projection maps $\pi_H\colon L\to \prod_{i\in H}$, defined by 
$u\mapsto \restrict u H$, and  $\pi_J\colon L\to \prod_{i\in J}$, defined by $u\mapsto \restrict u J$, are homomorphisms. Let $L_H:=\pi_H(L)$ and  $L_J:=\pi_J(L)$. Since $L_H$ and $L_J$ are homomorphic images of $L$, both are three-generated.  We know from Hobby and McKenzie~\cite[Theorem 0.1]{hobbymckenzie} that a finitely generated algebra in a variety generated by a finite set of finite algebras is necessarily finite. This fact and  $L_H\in\HSP\set K$ yield that the three-generated lattice $L_H$ is finite. On the other hand, $L_J\in \var V$ is a homomorphic image of $\vfree{\var V}3$ and so $L_J$ is also finite. Using that the map (in fact, homomorphism) 
$L\to L_H\times L_J$, defined by $u\mapsto (\pi_H(u), \pi_J(u))=(\restrict u H, \restrict u J)$ is injective, it follows that $L$ is finite. In particular,  $\vfree{\var W}3$ is finite, as required.

Now that we know that $L$ is finite, Lemma~\ref{lemmacZTkglnTxp} applies. So from now on, $I$, $H$, and $J$ in \eqref{eqHzrNmrCtJlRg} and \eqref{pbxZhRnWsnpXq} are finite index sets and \eqref{eqHzrNmrCtJlRg} is an irredundant subdirect product with respect for a fixed three-element generating set $X=\set{x,y,z}$ of $L$.
Keeping \eqref{pbxZhRnWsnpXq} in mind, consider the map
\begin{equation}
\begin{aligned}
\phi\colon L \to \prod_{i\in H} (K/\muk) \times \prod_{i\in J} L_i,\text{ defined by } u\mapsto u' \text{ such} \\
\text{that, for $i\in I$, }\,\,u'(i)=
\begin{cases}
\block{u(i)}\muk,&\text{if }i\in H,\cr
u(i),&\text{if }i\in J.
\end{cases}
\end{aligned}
\label{eqlnGnbmtrbK}
\end{equation}
Clearly, $\phi$ is a lattice homomorphism. Since  $K/\muk$ and the $L_i$ for $i\in J$ are all in $\var V$, the lattice 
\begin{equation}
L':=\phi(L)=\set{u': u\in L}\text{ belongs to }\var V.
\label{eqhzRmsZstBQxC}
\end{equation} 
Since $L'$ is defined as the $\phi$-image of $L$, the map $\phi\colon L\to L'$ is a surjective lattice homomorphism, 
whereby $L'$ is also a three-generated lattice. 

Now, we are in the position to prove part \eqref{keylemmaa} of Lemma~\ref{keylemma}. We assume that $K\notin \var V$ since otherwise the statement is trivial. 
Note that $\set{1,2,3}\subseteq \AS(\var V)$, because $\var V$ is a nontrivial variety of lattices, whence $\var D\subseteq \var V$ and \eqref{eqknXmpHla} applies. We need to show that $|\Ats L|\in \AS(\var V)$. There are two cases to consider. 

First, assume that $K$ satisfies meet condition \eqref{pbxmeetcond}. Based on Cz\'edli~\cite{czgnumatoms}, we can assume that \begin{equation}
\parbox{8cm}{at least two of $x\wedge y$, $x\wedge z$, and $y\wedge z$ are $0=0_L$;}
\label{eqpbxdzhGrmJrmrTkC}
\end{equation}
indeed, if  \eqref{eqpbxdzhGrmJrmrTkC} fails, then $|\Ats L|\in\set{2,3}$ by \cite[Observation 1.2.(ii)]{czgnumatoms} and so $|\Ats L|\in\AS(\var V)$, as required.
For the sake of contradiction, suppose that   $H\neq\emptyset$, see \eqref{pbxZhRnWsnpXq}, and let $i\in H$. Then $\pi_i$ from \eqref{eqZhhsPnkZrZtnk} is a surjective $L\to K$ homomorphism, and $\set{\pi_i(x),\pi_i(y),\pi_i(z)}=\set{x(i),y(i),z(i)}$ generates $K$. But this contradicts our recent assumption that $K$ satisfies meet condition \eqref{pbxmeetcond} since $\pi_i$ preserves \eqref{eqpbxdzhGrmJrmrTkC}.
This shows that $H=\emptyset$. Using that
\begin{equation}
\parbox{8cm}{$H=\emptyset$,
  \eqref{eqlnGnbmtrbK}, and \eqref{eqhzRmsZstBQxC}  lead to $L=L'\in\var V$,}
\label{eqpbxsWfmsjbnrBlG}
\end{equation}
we obtain that $|\Ats L|=|\Ats {L'}|\in \AS(\var V)$. 
We have settled the case when $K$ satisfies meet condition \eqref{pbxmeetcond}. 

Second, we assume that  $\muk$ is 0-separating but $K$ fails to satisfy meet condition \eqref{pbxmeetcond}. Since $K\notin\var V$ has been assumed, $K$ cannot be generated by less than three elements. Hence, we conclude from \eqref{pbxcbThcnDzWrtmN} that  $K$ is three-generated.
Assume also that $u\in L\setminus \set 0$, let $u':=\phi(u)$, and pick an index $i\in I$ such that $u(i)\neq 0$.  Either since $i\in H$ and $\muk$ is 0-separating, or since $i\in J$ and $u'(i)=u(i)$, \eqref{eqlnGnbmtrbK} yields that $u'(i)\neq 0$ and so $u'\neq 0$. Hence, the surjective homomorphism $\phi$ is 0-separating. Applying \eqref{pbxkenkVsPkhTjQwBh} to $\phi$, we obtain that 
\begin{equation}
\text{$\restrict\phi{\Ats L}\colon \Ats L\to \Ats{L'}$ is a bijective map.}
\label{eqHmrIkszGpTk}
\end{equation}
Since $L'\in\var V$ by \eqref{eqhzRmsZstBQxC},  \eqref{eqHmrIkszGpTk} above yields that  $|\Ats L|=|\Ats{L'}|\in \AS(\var V)$. This shows that $\AS(\var W)\subseteq \AS(\var V)$. Since the converse inclusion is a trivial consequence of $\var V\subseteq \var W$, we have proved part \eqref{keylemmaa}. 

Part \eqref{keylemmab} follows from part \eqref{keylemmaa} by duality.

Finally, the argument for  \eqref{keylemmac} also splits into two cases. Again, still assuming that $L\in \var W$, 
it suffices to show that $(|\Ats L,\Coats L|)\in \DS(\var V)$.
First, assume that $K$ satisfies both meet condition \eqref{pbxmeetcond}  and  join condition \eqref{pbxjoincond}. As a subcase, assume also that the index set $H$ is nonempty and pick an  $i\in H$. 
If the fixed generating set $X=\set{x,y,z}$ of $L$ satisfied \eqref{eqpbxdzhGrmJrmrTkC}, then $\set{\pi_i(x),\pi_i(y),\pi_i(z)}$ would be a generating set of $K$ and $\pi_i$ would preserve the equalities listed in  \eqref{eqpbxdzhGrmJrmrTkC}, but this would contradict that meet condition \eqref{pbxmeetcond} holds in $K$. Hence,  \eqref{eqpbxdzhGrmJrmrTkC} fails, whereby
Cz\'edli~\cite[Observation 1.2.(ii)]{czgnumatoms} gives that $|\Ats L|\in\set{2,3}$. 
Since  join condition \eqref{pbxjoincond} is also assumed, duality applies and we also have that
$|\Coats L|\in\set{2,3}$. Hence, $(|\Ats L|,|\Coats L|)\in\set{2,3}^2\subseteq \DS(\var D)\subseteq \DS(\var V)$, provided there is an $i$ in $H$. If there is no such $i$, then $H=\emptyset$ gives that $L\in \var V$ by \eqref{eqlnGnbmtrbK}, whence $(|\Ats L|,|\Coats L|)\in\DS(\var V)$ again, as required.

Second, assume that  $\muk $ is both $0$-separating and $1$-separating. Then, in addition to \eqref{eqHmrIkszGpTk}, we also have that
$\restrict\phi{\Coats L}\colon \Coats L\to \Coats{L'}$ is a bijective map by duality.
Hence, similarly to the proof of part \eqref{keylemmaa}, 
$(|\Ats L,\Coats L|) = (|\Ats{L'},\Coats{L'}|)\in \DS(\var V)$ as required.  The proof of Lemma~\ref{keylemma} is complete.
\end{proof}

Since the projection maps $\pi_i$ from \eqref{eqZhhsPnkZrZtnk} preserve equalities, the following remark is a trivial; we formulate it for later reference. 
\begin{remark}\label{remarkZfhrBmW} If $X=\set{x,y,z}$ is a fixed generating set of a subdirect product $L\spleq \prod_{i\in I}L_i$
such that  $(x,y,z)$ witnesses a failure of meet condition \eqref{pbxmeetcond} in $L$, then  so does $(\pi_i(x), \pi_i(y), \pi_i(z))$ in $L_i$ for every $i\in I$.  Hence, using that $x$, $y$, and $z$ play symmetric roles and based on the explanation around \eqref{eqpbxdzhGrmJrmrTkC}, we will  frequently assume that $x\wedge z=y\wedge z=0$ in $L$ and so
\begin{equation}
\pi_i(x)\wedge \pi_i(z)= \pi_i(y)\wedge \pi_i(z)=0\text{ in }L_i,\text{ for every }i\in I.
\label{eqdHtrsbmDnMsb}
\end{equation}
\end{remark}

\begin{figure}[ht]
\centerline
{\includegraphics[scale=1.0]{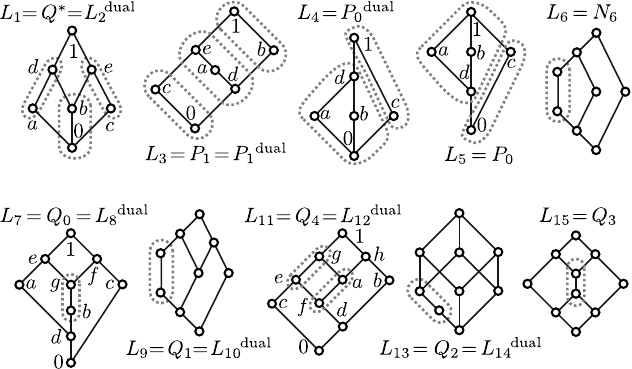}}
\caption{$L_1$, \dots, $L_{15}$, and their monolith congruences}
\label{figbb}
\end{figure}

\section{An interval in the lattice of all lattice varieties}\label{sect:intvl}

First of all, we need to recall some known concepts and notations and introduce some further notations.
Let $\var N_5$ denote the lattice variety $\HSP \set{N_5}$ generated by the pentagon lattice $N_5$; see Figure~\ref{figaa}. As usual, $M_3$ stands for the 5-element modular but not distributive lattice, see Figure~\ref{figaa} again, and we denote by $\var M_3$ the variety it generates. The dual of a lattice $L$ will be denoted by $\dual L$. When  dealing with elements of $\alllat$, that is, with lattice varieties, then $\prec$ stands for the covering relation understood in $\alllat$. In Figures~\ref{figbb} and \ref{figcc} (disregard the dotted ovals in the moment), we give the lattices playing the main role in this paper. Namely, $L_1$, ..., $L_{15}$ are taken from  McKenzie \cite{mckenzie} while the lattices $V_1$, \dots, $V_8$ from Ruckelshausen~\cite{ruckel}; see also Jipsen and Rose~\cite[Pages 19--20]{jipsenrose} for a secondary source. Note that in addition to McKenzie's original notations like $Q^*$ and $P_1$, Figure~\ref{figaa} gives the notations due to J\'onsson and Rival \cite{jonssonrival}. For $i\in\set{1,\dots,15}$ and $j\in\set{1,\dots,8}$, we let
\begin{align}\left.
\parbox{7.0cm}{
$\var N_5:=\HSP\set{N_5},\quad \var M_3:=\HSP\set{M_3}$, \\
$\kern 20pt \var L_i:=\HSP\set{L_i},\quad\var V_j:=\HSP\set{V_j}$, and, for any lattice variety $\var Z$, we let $\mvar Z:=\var Z\vee \var M$.}
\,\,\right\}
\label{pbxmskSgnKa}
\end{align}
Note that Jipsen and Rose~\cite[Page 21]{jipsenrose} denotes $\mvar N_5$ by $\var M^+$. In addition to this variety, \eqref{pbxmskSgnKa} defines twenty-three varieties larger than the variety $\var M$ of modular lattices; these varieties are the $\mvar L_i$ for $i\in\set{1,\dots, 15}$ and the $\mvar V_j$ for $j\in \set{1,\dots,8}$.

The following proposition is likely to belong to the folklore of lattice theory since it follows easily from widely known ideas.  Having no reference at hand, we are going to present a proof for it.

\begin{figure}[ht]
\centerline
{\includegraphics[scale=1.0]{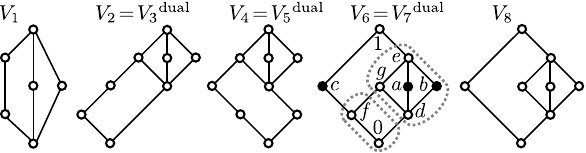}}
\caption{$V_1$, \dots, $V_{8}$, and the monolith congruence of $V_6$}
\label{figcc}
\end{figure}

\begin{proposition}\label{prophszHfd} In the lattice $\alllat$ of all lattice varieties, the following hold.  
\begin{enumeratei}
\item\label{prophszHfda} $\var M\prec \mvar N_5$ and $\mvar N_5$ is the only cover of $\var M$ in $\alllat$. Furthermore, for every $\var Y\in\alllat$, \   $\var Y\not\leq\var M$ implies that $\mvar N_5\leq \var Y$.
\item\label{prophszHfdb} For $i\in\set{1,\dots,15}$, we have that $\mvar N_5\prec \mvar L_i$. 
\item\label{prophszHfdc} For $j\in\set{1,\dots,8}$, we have that $\mvar N_5\prec \mvar V_j$. 
\item\label{prophszHfdd} The set $\set{\mvar L_i: 1\leq i\leq 15}\cup \set{\mvar V_j: 1\leq j\leq 8}$ is a $23$-element subset of $\alllat$ and it generates a sublattice isomorphic to the $2^{23}$-element Boolean lattice, which is an interval in $\alllat$; this sublattice is denoted by $\bhhrom$. 
\end{enumeratei}
\end{proposition}

Some elements of $\bhhrom$  are outlined in Figure~\ref{figdd}. We do not know whether all covers of $\mvar N_5$ are listed in Proposition~\ref{prophszHfd}.

\begin{proof}[Proof of Proposition~\ref{prophszHfd}]
First, we collect some known results that are needed.  McKenzie~\cite{mckenzie} conjectured and J\'onsson and Rival \cite{jonssonrival} proved that 
\begin{equation}\left.
\parbox{8cm}{$\var L_1$, \dots, $\var L_{15}$, and $\var N_5\vee \var M_3$ are sixteen distinct covers of $\var N_5$ in  $\alllat$. Furthermore, if $\var Y\in \alllat$ such that $\var N_5< \var Y$, then $\var Y$ includes at least one of these sixteen covers. Also, $\var L_i$ is join-irreducible in $\alllat$ for $i\in\set{1,\dots,15}$.}
\,\,\right\}
\label{pbxmckjrrGhT}
\end{equation} 
The join-irreducibility of $\var L_i$ is an easy consequence of 
B.\ J\'onsson's  \eqref{pbxBJg} and \eqref{pbxBJktt}, and it is explicitly mentioned in the last paragraph of page 18 in Jipsen and Rose~\cite{jipsenrose}.
Note that it follows from the second half of \eqref{pbxmckjrrGhT} that $\var N_5$ has \emph{exactly} sixteen covers. Ruckelshausen~\cite{ruckel}
proved that 
\begin{equation}\left.
\parbox{8cm}{$\var V_1,\dots,\var V_8$ are pairwise distinct join-irreducible elements of $\alllat$  and $\var M_3\vee \var N_5 \prec \var V_j\,\,$  for $\,\,j\in\set{1,\dots,8}$;}\,\,\right\}
\label{pbxrchsPdspKzg}
\end{equation}
see also Jipsen and Rose~\cite[Pages 19--20]{jipsenrose} for a secondary source. Unfortunately, we do not know whether $\var V_1,\dots,\var V_8$ is the list of \emph{all} covers or $\var N_5\vee \var M_3$.

Part \eqref{prophszHfda} of Proposition~\ref{prophszHfd} is trivial by Dedekind's modularity criterion. 

To prove  \eqref{prophszHfdb}, let $i\in\set{1,\dots,15}$. Since $\var N_5\prec \var L_i$ by \eqref{pbxmckjrrGhT}, the (upper) semimodularity of $\alllat$ yields that either $\mvar N_5=\var N_5\vee \var M = \var L_i\vee \var M=\mvar L_i$, or  $\mvar N_5\prec\mvar L_i$. We need to exclude the first alternative. For the sake of contradiction, suppose that  $\mvar N_5=\mvar L_i$. Since $L_i$ is subdirectly irreducible and $L_i\in \mvar L_i=\mvar N_5=\var N_5\vee \var M$,  \eqref{pbxBJg} gives that $L_i\in \var N_5$ or $L_i\in \var M$. This leads to $L_i\in \var N_5$ since  $L_i$ is not modular. Hence  \eqref{pbxBJktt} gives that $L_i\in\HSs\set{N_5}$, contradicting  $|L_i|>5=|N_5|$. This excludes the first alternative and proves part  \eqref{prophszHfdb}.

To prove  \eqref{prophszHfdc}, let  $j\in\set{1,\dots,8}$. Since $\var M_3\vee \var N_5\prec \var V_j$ by \eqref{pbxrchsPdspKzg}, the semimodularity of $\alllat$ gives that either $\mvar N_5= \var M\vee  \var N_5=  (\var M\vee \var M_3)\vee \var N_5 = \var M\vee (\var M_3\vee \var N_5) = \var M\vee \var V_j=\mvar V_j$, or $\mvar N_5 \prec \mvar V_j$. For the sake of contradiction, suppose  that  $\mvar N_5 =\mvar V_j$. Similarly to the previous paragraph, the subdirect irreducibility of $V_j$, \ $V_j\in \mvar V_j=\mvar N_5=\var N_5\vee\var M$,   \eqref{pbxBJg}, and   \eqref{pbxBJktt} give that $V_j\in\var M$ or $V_j\in\HSs\set{N_5}$, but this is a contradiction since $V_j$ is not modular and $|V_j[>|N_5|$.  We have excluded the first alternative and proved part \eqref{prophszHfdc}.

Observe that  $\var L_1$, \dots, $\var L_{15}$, as distinct covers of $\var N_5$,  are pairwise incomparable. Hence,   $L_i\notin \Si(\var L_{i'})$ if $i'\neq i$ and  
$\set{i,i'}\subseteq \set{1,\dots, 15}$. Also, $L_i\notin \var M$. Thus,
the subdirect irreducibility of $L_i$ and  
\eqref{pbxBJg} give that $L_i\notin \var L_{i'}\vee \var M=\mvar L_{i'}$,
whereby $\mvar L_i\not\leq \mvar L_{i'}$. Therefore,  the varieties $\mvar L_1$, \dots, $\mvar L_{15}$ are pairwise distinct. So are $\mvar V_1$, \dots, $\mvar V_{8}$ by an analogous reasoning. 
By \eqref{pbxmckjrrGhT} and \eqref{pbxrchsPdspKzg},  $\var V_j$ is of height 2 in the principal filter $\filter \var N_5$ of $\alllat$ 
but $\var L_i$ is only of height 1. Consequently,  $\var V_j\not\leq \var L_i$ and so  $V_j\notin \var L_i$. Also, $V_j$ is not in $\var M$ but it is subdirectly irreducible, whereby \eqref{pbxBJg} yields that $V_j\notin \var L_i\vee\var M=\mvar L_i$. Thus,  $\mvar V_j\not\leq \mvar L_i$. In particular,  $\mvar V_j\neq \mvar L_i$, and we conclude that 
 \eqref{prophszHfdd} presents a 23-element set, as required.
By the already proven \eqref{prophszHfdb} and \eqref{prophszHfdc}, this set consists of atoms of the filter $\filter \mvar N_5$.  

Next, we extract from the literature that, for any positive integer $n$, 
\begin{equation}
\parbox{10cm}{every $n$-element set of atoms of a distributive lattice generates a $2^n$-element Boolean sublattice and this sublattice is an interval.}
\label{pbxZrmTsjkQqG}
\end{equation}
To show this, let $D$ be a distributive lattice, let $a_1,\dots, a_n$ be pairwise distinct atoms of $D$, and let $S$ be the sublattice generated by $\set{a_1,\dots,a_n}$. 
If we had that 
$a_i\leq a_1\vee\dots\vee a_{i-1}$ for some $i\in\set{2,\dots,n}$, then distributivity would give that $a_i=a_i\wedge(a_1\vee\dots a_{i-1})=(a_i\wedge a_1)\vee\dots\vee(a_i\wedge a_{i-1})$.
But atoms are join-irreducible, whence $a_i=a_i\wedge a_j$ for some $j\in\set{1,\dots,i-1}$, that is, we would have that $a_i\leq a_j$, contradicting that $a_i$ and $a_j$ are distinct atoms. Hence,
$a_i\not\leq a_1\vee\dots a_{i-1}$ for all  $i\in\set{2,\dots,n}$. By 
the $\textup{(ii)}\Rightarrow \textup{(i)}$ and $\textup{(ii)}\Rightarrow \textup{(iii)}$ parts of Theorem 380 of 
 Gr\"atzer~\cite{ggfoundbook}, it follows that $\set{a_1,\dots, a_n}$ is an independent set of atoms and the height of $a_1\vee\dots a_n$ is $n$. Thus, by the definition of independence, $S$ is a boolean lattice of length $n$ and size $2^n$. By the structure theorem of finite distributive lattices, see Gr\"atzer~\cite[Theorem 107]{ggfoundbook}, a distributive lattice of length $n$ cannot have more than $2^n$  elements. Hence, all elements of the interval $[0, a_1\vee\dots\vee a_n]$ of $D$ belong to the $2^n$-element sublattice $S$. Therefore, this interval is $S$. 
This proves \eqref{pbxZrmTsjkQqG}. Note that \eqref{pbxZrmTsjkQqG} also follows from Cz\'edli~\cite[Proposition 2.1.(iv)]{czgcoord} since $D$ is a locally finite lattice. 

We have already seen that the set mentioned in \eqref{prophszHfdd} is a 23-element set of atoms in the filter $\filter\mvar N_5$. Thus, \eqref{pbxZrmTsjkQqG} implies part  \eqref{prophszHfdd} of Proposition~\eqref{prophszHfd}. 
\end{proof}

\section{The spectra of some lattice varieties}\label{sect:somespectra}
\begin{proposition}\label{propgMncmp}
 Let $\var W$ be a nontrivial variety of lattices and let
\begin{align}
\var B&:=\mvar V_8 \vee \bigvee_{i=6}^{15}   \mvar L_i \vee\bigvee_{j=1}^5 \mvar V_j\quad \text{ and }\quad\var B':= \var B\vee \mvar L_2\vee \mvar V_7.
\label{eqthmgPBMlPcsPdlha}
\end{align}
\begin{enumeratei}
\item\label{propgMncmpa} If $\var W\subseteq \var B$, then $\DS(\var W)$ equals $\DS(\var D)$, which is given in  \eqref{eqknXmpHld}.
\item\label{propgMncmpb} If $\var W\subseteq \var B'$, then $\AS(\var W)=\AS(\var D)=\set{1,2,3}$; see  \eqref{eqknXmpHla}.
\end{enumeratei}
Also, if $\var E$ is a lattice variety such that $\mvar N_5\leq \var E$ then
\begin{enumeratei}{\setcounter{enumi}2}
\item\label{propgMncmpd} $\DS(\var W)=\DS(\var E)$ for every lattice variety $\var W$ belonging to the interval $[\var E,\var E\vee \var B]$ of $\alllat$, and
\item\label{propgMncmpe} $\AS(\var W)=\AS(\var E)$ for every  $\var W\in [\var E,\var E\vee \var B']$.
\end{enumeratei}
\end{proposition}

\begin{proof} To prove part \eqref{propgMncmpa}, define $\var W_0:=\var M$,
\allowdisplaybreaks{
\begin{align*}
\var W_1&:=\mvar N_5= \HSP(\var W_0\cup\set{N_5}),\cr 
\var W_2&:=\mvar V_8=  \var W_1\vee\mvar V_8=\HSP(\var W_1\cup\set{V_8}),\cr
\var W_3&:=\mvar V_8\vee \mvar L_6=\var W_2\vee \var L_6 = \HSP(\var W_2\cup\set{L_6}),\cr
\var W_4&:=\mvar V_8\vee \mvar L_6\vee \mvar L_7=\var W_3\vee \var L_7 = \HSP(\var W_3\cup\set{L_7}), \dots,
\cr
\var W_{16}&:=\mvar V_8 \vee \bigvee_{i=6}^{15} \mvar L_i \vee\bigvee_{j=1}^4 \mvar V_j = \var W_{15}\vee \var V_4 = \HSP(\var W_{15}\cup\set{V_4})
,\cr
\var W_{17}&:=\mvar V_8 \vee \bigvee_{i=6}^{15} \mvar L_i \vee\bigvee_{j=1}^5 \mvar V_j = \var W_{16}\vee \var V_5 = \HSP(\var W_{16}\cup\set{V_5}).
\end{align*}}%
For later reference, let us point out that the order in the  list $V_8$, $L_6$, \dots, $V_4$, $V_5$ of lattices above is irrelevant in the sense that 
although the sequence $\var W_2$, \dots, $\var W_{17}$ depends on this order, any other order gives rise to a sequence of varieties that makes the rest of the proof work without any essential change.

Clearly, $\var W_0\subseteq \var W_1\subseteq \dots\subseteq\var W_{17}=\var B$. 
As \eqref{pbxcbThcnDzWrtmN} and Figures~\ref{figbb}--\ref{figcc} show, 
\begin{equation}\left.
\parbox{8.3cm}{the monoliths of the lattices $N_5$, $L_6$, $L_7$, \dots, $L_{15}$ occurring above are both 0-separating and 1-separating while $V_8$ and $V_1$, \dots, $V_5$ satisfy both meet condition \eqref{pbxmeetcond} and  join condition \eqref{pbxjoincond}.}\,\,\right\}
\label{pbxczhBmoTrkKZLps}
\end{equation}
All proper sublattices and homomorphic images of $N_5$ belong to $\var D$, so they belong to $\var W_0=\var M$. We claim that for every $i\in\set{1,\dots,15}$ and $j\in\set{1,\dots,8}$,
\begin{align}\left.
\parbox{7.7cm}{if $K$ is a proper homomorphic image or a proper sublattice of $L_i$, then $K\in\var N_5\subseteq\var W_1$;  }
\,\,\right\}
\label{pbxPrshMKrsTa} \\
\left.
\parbox{7.7cm}{if $K$ is a proper homomorphic image or a proper sublattice of $V_j$, then $K\in \var N_5\vee \var M_3\subseteq \mvar N_5=\var W_1$.}\,\,\right\}
\label{pbxPrshMKrsTb}	
\end{align}
Both \eqref{pbxPrshMKrsTa} and \eqref{pbxPrshMKrsTb} could be proved by inspecting lots of straightforward but tiring cases. Fortunately, \eqref{pbxmckjrrGhT} and \eqref{pbxrchsPdspKzg} permit a shorter proof. 
For the sake of contradiction, suppose that $K$ is a proper homomorphic image or a proper sublattice of $L_i$ but $K\notin \var N_5$. 
Note that $|K| < |L_i|$. With $\var X:=\var N_5\vee \HSP\set{K}$,
we have that $\var N_5<\var X$. On  the other hand, $\var N_5\leq \var L_i$  and  $K\in\HSs\set{L_i}\subseteq \HSP\set{L_i}= \var L_i$ give that
$\var X\leq \var L_i$. So $\var N_5<\var X\leq \var L_i$, and we conclude from \eqref{pbxmckjrrGhT} that $\var X=\var L_i$. Since $L_i$ is subdirectly irreducible, $L_i\in \var L_i=\var X=\var N_5\vee \HSP\set{K}$. Using \eqref{pbxBJg}, due to J\'onsson~\cite{jonsson}, we have that $L_i\in \var N_5=\HSP\set{N_5}$ or $L_i\in \HSP\set{K}$.  Hence, by 
\eqref{pbxBJktt}, $L_i\in \HSs\set{N_5}$ or $L_i\in \HSs\set{K}$. This gives that $|L_i|\leq \max\set{|N_5|, |K|}$, which is a contradiction proving \eqref{pbxPrshMKrsTa}.

With less details, the proof of  \eqref{pbxPrshMKrsTb} runs similarly as follows. Suppose that  \eqref{pbxPrshMKrsTb} fails. Pick a  proper homomorphic image or a proper sublattice $K$ of $V_j$ such that $K\notin \var N_5\vee \var M_3$.
With $\var X:= \var N_5\vee \var M_3\vee \HSP\set K$, we have that
$\var N_5\vee \var M_3 <\var X \leq \var V_j$. Using  \eqref{pbxrchsPdspKzg}, we obtain that $\var X =\var V_j$. This leads to 
$V_j\in \var V_j=\var X= (\var N_5\vee \var M_3)\vee \HSP\set K$. By \eqref{pbxBJg} and \eqref{pbxBJktt}, $V_j\in \var N_5\vee \var M_3$ or $V_j\in \HSs\set K$, but the first alternative contradicts  \eqref{pbxrchsPdspKzg} while the second to $|V_j|>|K|$. This yields the validity of  \eqref{pbxPrshMKrsTb}. 
\nothing{The proof of  \eqref{pbxPrshMKrsTb} is almost the same;  the differences are the following. Instead of $L_i$, $\var L_i$, \eqref{pbxmckjrrGhT}, $\var N_5$, and $\set{N_5}$, we need to use $V_j$, $\var V_j$,  \eqref{pbxrchsPdspKzg}, $\var N_5\vee \var M_3$, and $\set{N_5,M_3}$, respectively. Also, instead of $|L_i|\leq \max\set{|N_5|, |K|}$, the contradiction we obtain now is $|V_j|\leq \max\set{|N_5|,|M_3|, |K|}$. Thus, \eqref{pbxPrshMKrsTb} holds.}

Since $\afree{\var M}3$, consisting of 28 elements, is finite, it follows from Lemma~\ref{lemmakdlBRbFnW} that free lattice $\afree{\var W_i}3$ is also finite for every $i\in\set{0,1,\dots,17}$.
Armed with \eqref{pbxczhBmoTrkKZLps}, \eqref{pbxPrshMKrsTa}, and \eqref{pbxPrshMKrsTb},  we can apply Lemma~\ref{keylemma}\eqref{keylemmac} and, at the last step, \eqref{eqknXmpHld} to obtain that  
$\DS(\var B)=\DS(\var W_{17})=\DS(\var W_{16})=\dots =\DS(\var W_{0})=\DS(\var M)=\DS(\var D)$. So $\DS(\var B)=\DS(\var D)$. This together with  $\var D\leq \var W$ and  \eqref{eqknXmpHle} prove part \eqref{propgMncmpa}.

The monolith of $L_2$ and that of $V_7$ are 0-separating. Using these two lattices in the same way as the earlier ones, we can continue the sequence $\var W_0,\dots \var W_{17}$ with
$\var W_{18}=\HSP(\var W_{17}\cup\set{L_2})$ and 
$\var B'=\var W_{19}=\HSP(\var W_{18}\cup\set{V_7})$. Now, instead of Lemma~\ref{keylemma}\eqref{keylemmac}, we can apply  Lemma~\ref{keylemma}\eqref{keylemmaa}. In this way, we obtain the validity of part \eqref{propgMncmpb} in the same way as that of  \eqref{propgMncmpa}.

Next, to prove part  \eqref{propgMncmpd}, assume that $\var E\geq\mvar N_5$.
We can even assume that $\var E > \mvar N_5$ since otherwise the already proven  part \eqref{propgMncmpa}  would apply.
Let $\var F:=\var B\wedge \var E$, and observe that $\var F$ belongs to the interval $[\mvar N_5, \var B]$ of $\alllat$. 
 It follows from \eqref{eqtxtBnrVl} or, rather say, from Proposition~\ref{prophszHfd}\eqref{prophszHfdd}  that $\var F\in \bhhrom$.  Hence, $\var F$ is the join of $\mvar N_5$ and some of the varieties belonging to  the set $\set{\mvar V_8,\mvar L_6,\dots, \mvar V_5}$ that occur in the definition of $\var B$ in \eqref{eqthmgPBMlPcsPdlha}. Therefore, based on the sentence following the definition of $\var W_{17}$, we can assume that $\var F=\var W_j$ for some $j\in\set{1,\dots, 17}$. 
To ease the notation, we let $j=2$; the general case is the practically the same. Take the sequence $\var W_2':=\var E\vee \var W_2=\var E\vee \var F=\var E$, \ 
$\var W_3':=\var E\vee \var W_3=\HSP(\var W'_2\cup\set{L_6})$, \ 
$\var W_4':=\var E\vee \var W_4=\HSP(\var W'_3\cup\set{L_7})$, \ 
$\var W_5':=\var E\vee \var W_5=\HSP(\var W'_4\cup\set{L_8})$, \dots,
$\var W_{17}':=\var E\vee \var W_{17}=\HSP(\var W'_{16}\cup\set{V_5})$. 
Since $\var W_{17}=\var B$, our sequence terminates with $\var W'_{17}=\var E\vee \var B$. The argument used in \eqref{propgMncmpa} applies for the sequence 
$\var W_2'$, \dots, $\var W_{17}'$ and yields that 
$\DS(\var E\vee \var B)=\DS(\var W'_{17})=\DS(\var W'_{16})=\dots =\DS(\var W'_{2})=\DS(\var E)$. Hence, applying \eqref{eqknXmpHle}, we obtain the validity of part  \eqref{propgMncmpd}. 

Finally, the proof of part \eqref{propgMncmpe} is obtained by modifying that of 
 \eqref{propgMncmpd} in the same straightforward way as we modified the proof of part  \eqref{propgMncmpa} to obtain that of  \eqref{propgMncmpb}. 
The proof of Proposition~\ref{propgMncmp} is complete.
\end{proof}

\begin{remark}\label{rem:hnHblGbrTnhSprc}  
The rudiments of the theory of distributive lattices yield that the interval $[\mvar N_5, \var B']$ of $\alllat$ is (isomorphic to) the $2^{18}$-element boolean lattice; see also the proof of Proposition~\ref{prophszHfd}\eqref{prophszHfdd}.  Therefore,  Proposition~\ref{propgMncmp}\eqref{propgMncmpb} extends the scope of \eqref{eqknXmpHla} from Cz\'edli~\cite{czgnumatoms} with $2^{18}$ new lattice varieties. 
Similarly, Proposition~\ref{propgMncmp}\eqref{propgMncmpa} adds $2^{16}$ new lattice varieties to the scope of the previously known \eqref{eqknXmpHld}.
\end{remark}

Next, we conclude this section by exceeding \eqref{eqknXmpHlb}, an earlier result.  In the direct cube of $L_4=\dual{P_0}$, see Figure~\ref{figbb}, let 
\begin{equation}
\vx:=(c,a,a),\quad\vy=(a,c,b),\quad \vz=(b,b,c),
\end{equation}
and let $L_{92}$ be the sublattice of the direct cube ${L_4}^3$ generated by $\set{\vx, \vy, \vz}$. In fact, $L_{92}$ is a subdirect power of $L_4$. We are in the position to present the following observation; its proof will be given after a remark.

\begin{observation}\label{observ92}
The lattice $L_{92}$ is three-generated and it has at least six atoms.
\end{observation} 

\begin{remark}
The proof to be given soon has also been checked by a computer program.
Hence, we know that  $|L_{92}|=92$, explaining the notation, and $L_{92}$  has \emph{exactly} six atoms.   
\end{remark}

\begin{proof}[Proof of Observation~\ref{observ92}] For
brevity, we write triples without commas and parentheses; for example, $caa$ stands for $\vx=(c,a,a)$. Let us compute:
\begin{align}
0aa&=\vx\wedge d11= \vx\wedge (\vy\vee\vz)\in  L_{92},\cr
a0b&=\vy\wedge 1d1=\vy\wedge(\vx\vee \vz)\in  L_{92},\cr
bb0&=\vz\wedge 11d = \vz\wedge (\vx\vee \vy)\in  L_{92},\cr
a00&=\vy\wedge 1da=\vy\wedge(\vx\vee bb0)\in L_{92},\label{eqztmDhQa}\\
0a0&=\vx\wedge d1b=\vx\wedge(\vy\vee bb0)\in L_{92},\label{eqztmDhQb}\\
00a&=\vx\wedge db1=\vx\wedge (\vz\vee a0b)\in L_{92},\label{eqztmDhQc}\\
b00&=\vz\wedge 1ad =\vz\wedge (\vx\vee a0b)\in L_{92},\label{eqztmDhQd}\\
0b0&=\vz\wedge a1d =\vz\wedge (\vy \vee 0aa)\in L_{92},\label{eqztmDhQe}\\
00b&=\vy\wedge bd1 =\vy\wedge (\vz\vee 0aa)\in L_{92}.\label{eqztmDhQf} 
\end{align}
Since the elements \eqref{eqztmDhQa}--\eqref{eqztmDhQf} are atoms even in the direct cube ${L_{4}}^3$, they are also atoms in $L_{92}$.
\end{proof}

Note that Observation~\ref{observ92} explains why Lemma~\ref{keylemma}\eqref{keylemmaa} contains the stipulation that $K$ should satisfy meet condition \eqref{pbxmeetcond} or its monolith should be 0-separating. Also, the definition of $L_{92}$ together with Observation~\ref{observ92} show that $L_4$ cannot occur among the joinands in \eqref{eqthmgPBMlPcsPdlha}.

\section{More about the spectra of varieties belonging to $\bhhrom$}\label{sectionmorabout}
Recall that $\bhhrom$ and some  related notations are defined in \eqref{pbxbshrBhrm}, \eqref{pbxthTlvRhrM}, and \eqref{pbxmskSgnKa}. Here we are going to have a closer look at the structure of $\bhhrom$.
The sixteen vertices in  Figure~\ref{figdd} form a sublattice of $\bhhrom$. The solid line segments in the figure denote coverings in $\bhhrom$ while the dotted and dashed line segments stand for intervals having  more than two elements in $\bhhrom$. 
By Proposition~\ref{prophszHfd}\eqref{prophszHfdd}, $\bhhrom$ is an interval of $\alllat$, 
\begin{align}
&\Ats{\bhhrom}:=\set{\mvar L_1,\dots,\mvar L_{15}}\cup\set{\mvar V_1,\dots, \mvar V_8}, \,\, \text{ and}
\label{eqmnFkTrgwsPmZldSt}\\
&\left.\parbox{7.2cm}{each element of $\bhhrom\setminus\set{\mvar N_5}$ can be given as the join of a unique nonempty subset of $\Ats{\bhhrom}$.}\,\,\right\}
\label{eq:lnMgtDfmtHLr}
\end{align}
Equivalently, a lattice variety $\var X\in \bhhrom$ is uniquely determined by its intersection with  $\set{L_1,\dots, L_{15}}\cup\set{V_1,\dots, V_8}$. 
In harmony with Figure~\ref{figdd}, we let
\begin{equation}
\mvar H_3:=\mvar L_1\vee\mvar L_5\vee \mvar V_6\qquad\text{ and }\qquad \mvar H_6:=\mvar H_3\vee \mvar L_3\vee \mvar L_4.
\label{eqzMywRdlWtzls}
\end{equation}
The interval $[\mvar N_5,\mvar H_3]$ is isomorphic to the eight-element boolean lattice. So are seven other intervals in the figure that are transposed to  $[\mvar N_5,\mvar H_3]$; these eight intervals are drawn  by 
dotted northwest--southeast oriented lines segments 
 and we call them \emph{dotted intervals}.  
 The interval  $[\mvar N_5,\mvar H_6]$ is green-filled in Figure~\ref{figdd};  we call it the \emph{green  interval}. The green interval is  
isomorphic to the 32-element boolean lattice. Furthermore,  it is the union of four pairwise disjoint  dotted intervals:  
\begin{equation}[\mvar N_5,\mvar H_6]=
[\mvar N_5,\mvar H_3]\mathrel{\dot\cup}[\mvar L_3,\mvar L_3\vee \mvar H_3] \mathrel{\dot\cup} [\mvar L_4,\mvar L_4\vee \mvar H_3] \mathrel{\dot\cup} [\mvar L_3\vee \mvar L_4, \mvar H_6];
\label{eqVlzrbTssGma}
\end{equation} 
note that, in this paper,   $\dot\cup$ means that we form the union of \emph{pairwise disjoint} sets. 
The interval $[\mvar N_5,\var B']$ and seven other intervals transposed to it are drawn by dashed line segments; we call them \emph{dashed intervals}. As in Remark \ref{rem:hnHblGbrTnhSprc}, we can see that each dashed interval is a boolean lattice of size $2^{18}$.
Note that, by \eqref{eq:lnMgtDfmtHLr}, 
\begin{equation}\left.
\parbox{9.8 cm}{every member of $\bhhrom$ can uniquely be written in the form $\var X\vee\var Y$ such that $\var X$ is in the green interval and $\var Y\in [\mvar N_5,\var B']$.}\right\}
\label{eq:zlnkgchmZbrJng} 
\end{equation}
In fact, 
 $\bhhrom\cong  [\mvar N_5,\mvar H_6] \times  [\mvar N_5,\var B']$.     We have that 
\begin{equation}\left.
\begin{aligned}
\bhhrom={}
&
[\mvar N_5,\mvar H_3\vee \var B']   \mathrel{\dot\cup} 
[\mvar L_3,\mvar H_3\vee \mvar L_3\vee \var B']  \cr
 \mathrel{\dot\cup}{} &
[\mvar L_4,\mvar H_3\vee \mvar L_4\vee \var B'] \mathrel{\dot\cup}
[\mvar L_3\vee\mvar L_4,\mvar H_6\vee \var B'].
\end{aligned}\,\,\right\}
\label{eqVlzrbTssGmb}
\end{equation}
According to Figure~\ref{figdd},  the four intervals occurring in \eqref{eqVlzrbTssGmb} are called the \emph{layers} of $\bhhrom$. They are $2^{21}$-element boolean lattices, and each of them has its own fill pattern in the figure. Hence, as  the column ``Notation'' on the right of the figure indicates, each layer is colored by one of the numbers 3, 4, and 6. Namely, in the order they occur in \eqref{eqVlzrbTssGmb}, the layers are 
3-colored, 4-colored, 6-colored, and 6-colored, respectively.  
We define the color of a variety $\var X\in \bhhrom$ as follows: 
\begin{equation}
\text{the \emph{color} of $\var X\in\bhhrom$ is the color of the layer containing  $\var X$;}
\label{pbxZsgRmszTLkPk}
\end{equation} 
it follows from  \eqref{eqVlzrbTssGmb} that the color of $\var X$ is uniquely defined. Since we have also defined the colors with reference to \eqref{eqVlzrbTssGmb}, it is worth noting that the color of  $\var X\in\bhhrom$ has also been defined without referring to Figure~\ref{figdd}.

Now, keeping Figure~\ref{figdd} and the notations \eqref{eqlHzfQa}--\eqref{eqlHzfQc}, \eqref{pbxbshrBhrm},  \eqref{pbxmskSgnKa}, \eqref{eqthmgPBMlPcsPdlha}, and \eqref{eqzMywRdlWtzls}--\eqref{pbxZsgRmszTLkPk} in mind and introducing the notations
\begin{align}
\delta(3)&:=\{(1,1),(1,2), (2,1), (2,2), (2,3), (3,2), (3,3)\},\\
\delta(4)&:=\delta(3)\cup\set{(4,3), (3,4), (4,2), (2,4)},\\
\delta(6)&:=\delta(3)\cup\set{(4,2), (4,3), (6,3)},\text{ and}\\
\delta(6)^{-1}&:=\delta(3)\cup\set{ (2,4), (3,4), (3,6)},
\end{align}
we are in the position to formulate the main result of the present paper.

\begin{theorem}[Main Theorem]\label{thmmain}
Let $\var X$ be a lattice variety belonging to $\bhhrom$.
\begin{enumeratei}
\item\label{thmmaina} 
If the color of $\var X$, with respect to \eqref{pbxZsgRmszTLkPk} is $3$, $4$, or $6$, then the atom spectrum $\AS(\var X)$ of $\var X$ is $\set{1,2,3}$,  $\set{1,2,3,4}$, and $\set{1,2,3,4,6}$, respectively.  
\item\label{thmmainb} 
If $\var X$ belongs to one of the following intervals, then its double spectrum is given by the table below.
\[
\vbox{\tabskip=0pt\offinterlineskip
\halign{\strut#&\vrule#\tabskip=2pt plus 2pt&
#\hfill& \vrule\vrule\vrule#&
\hfill#&\vrule\tabskip=0.1pt#&
#\hfill\vrule\vrule\cr
\vonal\vonal\vonal\vonal
&&\hfill If $\var X$ belongs to the interval below,&&then $\DS{(\var X)}$ is\hfill&
\cr\vonal\vonal
&&\hfill$[\mvar N_5,\var B]$&&$\delta(3)$\hfill&
\cr\vonal\vonal
&&\hfill$[\mvar V_6,\mvar V_6\vee \var B]$ or $[\mvar V_7,\mvar V_7\vee \var B]$ &&$\delta(3)$\hfill&
\cr\vonal\vonal
&&\hfill$[\mvar L_1,\mvar L_1\vee \var B]$ or $[\mvar L_2,\mvar L_2\vee \var B]$&&$\delta(3)$\hfill&
\cr\vonal\vonal
&&\hfill$[\mvar L_3,\mvar L_3\vee \var B]$&&$\delta(4)$\hfill&
\cr\vonal\vonal
&&\hfill$[\mvar L_4,\mvar L_4\vee \var B]$&&$\delta(6)$\hfill&
\cr\vonal\vonal
&&\hfill$[\mvar L_5,\mvar L_5\vee \var B]$&&$\delta(6)^{-1}$\hfill&
\cr\vonal\vonal
\vonal\vonal
}} 
\]
\end{enumeratei}
\end{theorem}

Before the proof, some remarks are appropriate here. Surprisingly, while 6 belongs to the atom spectrum of some $\var X\in\bhhrom$, the number  $5$ does not. The ``coatom spectrum counterpart'' of part \eqref{thmmaina}, which we do not formulate in the present paper, would follows easily by duality. The intervals listed in part \eqref{thmmainb} are pairwise disjoint. However, their union is much smaller than $\bhhrom$. Actually, this union consists of $2^{21}$ elements, whereby part \eqref{thmmainb} takes care only a quarter of lattice varieties belonging to $\bhhrom$.

\begin{proof}[Proof of Theorem~\ref{thmmain}] 
In order to prove part \eqref{thmmaina}, it suffices to prove that 
\begin{equation}\left.
\parbox{8cm}{the atom spectra of the eight members of 
the green interval $[\mvar N_5, \mvar H_6]$ that are indicated by vertices in Figure~\ref{figdd} are the same as stated in the theorem.}\,\,\right\}
\label{pbxbNlCsKrldFzZGh}
\end{equation}
Indeed, if \eqref{pbxbNlCsKrldFzZGh} held, then \eqref{eqknXmpHle} and \eqref{eqVlzrbTssGma} would yield the validity of part \eqref{thmmaina} for all the 32 varieties belonging to the green interval $[\mvar N_5, \mvar H_6]$. Thus, 
it would follow from
\eqref{eq:zlnkgchmZbrJng} and 
 and Proposition~\ref{propgMncmp}\eqref{propgMncmpe} 
that any two varieties belonging to the same 
layer have the same atom spectrum, whereby \eqref{pbxbNlCsKrldFzZGh} would  imply the validity of part \eqref{thmmaina}.

We already know from Proposition~\ref{propgMncmp}\eqref{propgMncmpb}  that part  \eqref{thmmaina} of Theorem~\ref{thmmain} describes $\AS(\mvar N_5)$ correctly. So the job for one of the eight varieties mentioned in \eqref{pbxbNlCsKrldFzZGh} is done.  For the seven other varieties, both theoretical considerations and the brutal force of a computer are needed. We give the theoretical consideration only for $\mvar H_6$ since the rest of the seven varieties can be treated in an analogous but easier way. (In fact, we do not have to deal with all of them since \eqref{eqknXmpHle} applies in some cases.)
We claim that, up to isomorphism, 
\begin{equation}\left.
\parbox{9cm}{the set $\thrSi(\mvar H_6)$ of at most three-generated subdirectly irreducible lattices of $\mvar H_6$ is   $\set{\chain2, M_3, N_5, L_1, L_5, V_6, L_3, L_4}$.}\,\,\right\}
\label{pbxGrPfWtTwrQ}
\end{equation}
(Here ``up to isomorphism'' means that $\thrSi(\mvar H_6)$ is actually the set of \emph{isomorphism types} of the class of the at most three-generated subdirectly irreducible lattices of $\mvar H_6$, but it will be more convenient to work with this eight-element set than a proper class.) 
To show the validity of \eqref{pbxGrPfWtTwrQ}, let $K\in\mvar H_6$ be a subdirectly irreducible lattice generated by at most three elements. By \eqref{pbxmskSgnKa},  \eqref{eqzMywRdlWtzls}, and  Bjarni J\'onsson's result \eqref{pbxBJg},
\begin{equation}
\begin{aligned}
K\in \Si(\var M)&\cup\Si(\HSP\set{L_1}) \cup\Si(\HSP\set{L_5}) \cr
&\cup\Si(\HSP\set{V_6}) \cup\Si(\HSP\set{L_3})  \cup\Si(\HSP\set{L_4}).
\end{aligned}
\label{eqsizGrsrBkDt}
\end{equation} 
According to \eqref{eqsizGrsrBkDt}, the argument splits into three cases.
First, if $K\in \Si(\var M)$, then $K\in\set{\chain2,M_3}$ and \eqref{pbxGrPfWtTwrQ} is clear. Second, assume that  $K\in \Si(\HSP\set{V_6})$. Then  \eqref{pbxBJktt} gives that 
$K\in \HSs\set{V_6}$. We can assume that $K\not\cong V_6$ since otherwise  \eqref{pbxGrPfWtTwrQ} is clear. Hence, $K\in\HSP\set{K'}$ for a proper homomorphic image or  proper sublattice $K'$ of $V_6$.  
The lattice $K'$ is nontrivial since $K\in\HSP\set{K'}$  and $K$, being subdirectly irreducible, has at least two elements. Hence, $\var D\leq \HSP\set{K'}$ and so $\HSP\set{K'}\in [\var D, \var V_6]$. Using \eqref{pbxrchsPdspKzg}, the semimodularity of $\alllat$, and the obvious coverings $\var D\prec \var M_3$ and $\var D\prec \var N_5$, we conclude easily that the interval $[\var D, \var V_6]$ is of length 3. Clearly, $V_6\notin \HSs\set{K'}$ since $|V_6|>|K'|$. Applying   \eqref{pbxBJktt} again, we obtain that $V_6\notin \HSP\set{K'}$. Hence, 
$\HSP\set{K'}<\HSP\set{V_6}=\var V_6$. Using that $\var V_6$ is a join-irreducible cover of $\var M_3\vee \var N_5$ by \eqref{pbxrchsPdspKzg} and that both $\HSP\set{K'}$ and $\var M_3\vee \var N_5$  are in the interval $[\var D, \var V_6]$ of (finite) length 3,  we conclude that $\HSP\set{K'}\leq \var M_3\vee \var N_5$.
So $K\in \var M_3\vee \var N_5=\HSP \set{M_3}\vee \HSP \set {N_5}$, and the validity of  \eqref{pbxGrPfWtTwrQ} in this case follows easily by  
 \eqref{pbxBJg} and  \eqref{pbxBJktt}.
Third, assume that $K\in \Si(\HSP\set{L_i})$ for some $i\in\set{1,5,3,4}$. In the same way as in the second case above but using \eqref{pbxmckjrrGhT} and 
$[\var D, \var L_i]$ (of length 2) instead of \eqref{pbxrchsPdspKzg} and $[\var D, \var V_6]$, we obtain that $K\not\cong L_i$ implies that 
$K\in\HSP\set{N_5}$, and 
and the validity of  \eqref{pbxGrPfWtTwrQ} in this case follows immediately from \eqref{pbxBJktt}. This completes the proof of \eqref{pbxGrPfWtTwrQ}.

Next, still only focusing on $\mvar H_6$, we begin to season the theoretical consideration by computational aspects.
Let $L\in\mvar H_6$ be a three-generated lattice and fix a three-element generating set $\set{x,y,z}$ of $L$. It follows from Lemma~\ref{lemmacZTkglnTxp} that, in the sense of \eqref{pbxRsPchtmrGrthH} and up to isomorphism, $L$ is an irredundant subdirect product of a system of subdirectly irreducible lattices taken from  $\thrSi(\mvar H_6)$, which is described by \eqref{pbxGrPfWtTwrQ}. Even if we only allow irredundant subdirect products, a lattice from $\thrSi(\mvar H_6)$ can occur, with different assignments of the generators, more than once in the product.  Let us determine what is the maximum number of factors in an
irredundant subdirect product if multiplicities are counted. 
Since $\chain 2$ has no nontrivial automorphism and since singleton factors can be disregarded, there are six ways to pick a triplet in $\chain 2^3$ with components generating $\chain 2$. Hence, six copies of $\chain 2$ are needed.  We need only one copy of $M_3$ due to its large automorphism group $\Aut(M_3)$. 
Using that $|\Aut(N_5)|=1$, $|\Aut(L_1)|=2$, 
$|\Aut(L_5)|=2$, $|\Aut(V_6)|=2$, $|\Aut(L_3)|=1$, and $|\Aut(L_4)|=2$, 
we need $3=6/2$ copies of each of $L_1$, $L_5$, $V_6$, and $L_4$ but we need 6 copies of $N_5$ and the same number of copies of $L_3$. Hence, unless no reduction was found,
\begin{equation}\left.
\parbox{8.1cm}{we would have to work in a direct product of $6+1+4\cdot 3+6 +6=31$ factors and the size of this direct product would be $p_1:=2^6\cdot 5\cdot 7^3\cdot 6^3 \cdot 9^3 \cdot 6^3 \cdot 7^6 \cdot 5^6= 6\,862\,579\,602\,459\,840\,000\,000\approx 6.86\cdot 10^{21}$.}\,\,\right\}
\label{pbxlScskNts}
\end{equation} 
The number $6.86\cdot 10^{21}$ is too large, and what is also too bad is that we would have to take nonempty subset of the set of 31 factors in all possible ways, that is, in $2^{31}-1$ ways. That much computation is not feasible. In \eqref{pbxlScskNts}, we have taken into account only those criticizing homomorphisms of \eqref{pbxRsPchtmrGrthH} that are isomorphisms. Those that are not isomorphisms give some reduction but not enough. Typically, only $\chain2$ is a homomorphic image of another factor but even all the six copies of $\chain2$ were excluded, still 25 factors would remain, the largest direct product would be of size 
\begin{equation}
\text{$p_1\cdot 2^{-6}\approx 1.07\cdot 10^{20}$, and $2^{25}-1=33\,554\,431$}\label{eqtZrhNSzsrg}
\end{equation}  
products should be investigated, which would not be feasible.  Indeed, it is not sufficient to decide which subdirect products are needed, those that are needed have to be  constructed. 
If we  deal with double spectra, then we do not know further ways of reducing computations; this explains that Theorem~\ref{thmmain}\eqref{thmmainb} only deals  with a quarter of $\bhhrom$. 

If we deal with atomic spectra, then we are lucky enough to disregard  several subdirect factors by \eqref{eqdHtrsbmDnMsb}. For example, if $N_5$ is the $i$-th subdirect factor, then we can stipulate that $z(i)=c$, see Figure~\ref{figaa} for the notation of the elements of $N_5$, because otherwise   \eqref{eqdHtrsbmDnMsb}  would be violated.
Hence, instead of six, only two copies of $N_5$ are sufficient in the list of subdirect factors, and the same holds for $L_3$ and $V_6$. 
(In fact, one copy of $V_6$ is sufficient since the other one was excluded when calculating \eqref{pbxlScskNts}.)

To do the hard computations, we have written a computer program in
Bloodshed \emph{Dev-Pascal} v1.9.2 (Freepascal) under Windows 10 operating system, which is available from the author's website. This program takes its input from two distinct text files. The first file contains the operation tables of the possible subdirectly irreducible factors. (We have also written an auxiliary program that produces this file from the covering graphs of the irreducible factors.)
With reference to the first file, the second file gives the assignments of the variables and it can also give  constraints. 

The list of \emph{assignments} only takes care a part of irredundancy, see \eqref{pbxRsPchtmrGrthH} and Lemma~\ref{lemmacZTkglnTxp}. Namely, criticizing isomorphisms among distinct factors are excluded by the list of these assignments but non-bijective  criticizing homomorphisms are permitted. If the atom spectrum is targeted, then \eqref{eqdHtrsbmDnMsb} 
is also taken care of by the list of assignments: only those assignments occur in the list that obey  \eqref{eqdHtrsbmDnMsb}.
The assignments are given in so-called \emph{assignment lines}. 
These lines  follow strict syntactical rules but they are self-explanatory to read by humans. For example, if the third assignment line is\\
\verb!\lattice=N5 \with x=b y=a z=c!\\
then the program assumes that $N_5$ is the third subdirectly irreducible factor and lets $x(3):=b$, $y(3):=a$, and $z(3):=c$. 

The purpose of a \emph{constraint} is to instruct the program to take a non-bijective criticizing homomorphism into account. 
The constraints, if any, are given by so-called  \emph{constraint lines}. Constraint lines are self-explanatory again;  for example,\\
\verb!\if N5 \with x=b y=a z=c \ThenNot C2 \with x=0 y=0 z=1!\\
is one of the constraint lines when dealing with $\AS(\mvar H_6)$. 

The terminological difference between assignments and constraints is explained by their different roles in the program. 
The second file for $\AS(\mvar H_6)$ consists of 
seventeen assignment lines and nineteen constraint lines. 
The program takes all the $2^{17}-1$ subsets $I$ 
of the set of assignment lines one by one. For each $I$, the program  verifies whether all the constraints are satisfied. If they are, then the program constructs the subdirect product of
the subdirectly irreducible lattices belonging to $I$ and counts its atoms. On a desktop computer with AMD Ryzen 7 2700X Eight-Core Processor 3.70 GHz, it took hardly more than 5 minutes to obtain $\AS(\mvar H_6)$. 

The computations for the proper subvarieties of $\mvar H_6$ that are indicated by  vertices in Figure~\ref{figdd} took less than 
six seconds. We have outlined the proof of part \eqref{thmmaina} of Theorem~\ref{thmmain}, on which the computer spent about six minutes.

The proof of part  \eqref{thmmainb} is practically the same
; we only give the differences. The main difference is that we cannot use \eqref{eqdHtrsbmDnMsb}. Hence, for example, we cannot reduce what \eqref{eqtZrhNSzsrg} says about  $\mvar H_6$. This is why we could use our computer program only for the bottom elements of the intervals in the table that goes with part \eqref{thmmainb} of Theorem~\ref{thmmain}.
For comparison, note that while the program computed $\AS(\mvar L_3)$ in less than a second, it spent three and a half hours on computing $\DS(\mvar L_3)$.

Finally, as the last sentence of the proof we present here, we mention that the reader can access all details by downloading the computer program together with its input and output files from the author's website.
\end{proof}

\begin{remark}[on the program] There is an earlier program  developed by Berman and Wolk~\cite{bermanwolk} that is somehow related to ours. The two programs were written for different purposes in different programming languages for different computers in different times. Nevertheless, these two programs share some ideas. As mentioned already, 
Lemma~\ref{lemmacZTkglnTxp}, from which both programs benefit heavily, has been extracted from Berman and Wolk~\cite{bermanwolk}. 
This lemma is the only influence of Berman and Wolk's program on the present one; first because the programming language they use is not readable for me, second because the current program takes lots of
ideas and parts from my programs that go with Cz\'edli~\cite{czgpartlat} and Cz\'edli and Oluoch~\cite{czgLO}.

Both programs compute a subdirect product by generating the corresponding sublattice in a direct product. The core of both programs is to calculate the sublattice a given set $X_0$ generates. Even if there is a trivial algorithm, it has to be accelerated in two ways. First, instead of computing a sequence of subsets by the rule
\[X_{i+1}:=X_i\cup\set{x\vee y: (x,y)\in X_i^2}\cup\set{x\wedge y: (x,y)\in X_i^2}
\] 
as long as $X_i\supset X_{i-1}$ (proper inclusion), the program lets 
\begin{align}X_{i+1}:=X_i&\cup\set{x\vee y: (x,y)\in (X_{i}\setminus X_{i-1})\times X_i } \cr
&\cup \set{x\wedge y: (x,y)\in(X_{i}\setminus X_{i-1})\times X_i }. \label{eqlNgznMrMksz}
\end{align}
the improvement in speed is essential.
Second, even if the direct product of the subdirectly irreducible factors is very large and it cannot be stored (or not intended to be stored) by the program, the subdirect product what $X_0$ generates is much smaller in general and it should be stored. When a new element, say a new meet according to \eqref{eqlNgznMrMksz}, is computed, the program has to check whether it is already present in $X_i$ or it should really be added. 
Hence, according to a trivial algorithm, the program should check all elements of $X_i$ and compare them to the candidate new element; this needs $|X_i|$ steps for a new element and $|X_i|/2$ steps in average if the candidate element is not new. Even  $|X_i|/2$ is rather large since this activity has to be repeated very many times.  
To accelerate this trivial algorithm, Berman and Wolk~\cite{bermanwolk} used some hash function. Our program follows a different strategy: we store $X_i$ in a binary tree with the property that all elements of the left subtree of a node are lexicographically smaller while those of the right subtree are larger than the node in question. As a result, the program reduces the above-mentioned $|X_i|$ or $|X_i|/2$ steps down to $\log_2(|X_i|)$ steps. 
\end{remark}

\begin{remark}[on the reliability of the program]
Generally, a computer program is more difficult to check and it is more exposed to hidden errors than a mathematical proof. While  mathematical papers trust themselves, it is quite typical that a computer program declares itself by the words ``is as is''. 

Even if the  source file of the current program, called \texttt{atoms3}, is only 54 kilobytes while its auxiliary program, called \texttt{isitlatt}, is 26 kilobytes, and even if I have spent lots of time on testing the program as a whole and also its parts separately, the most satisfactory way of testing would be to use another program written by another person.
Fortunately, this has mostly been realized already in the following way. 
The critical part of the program is to compute subdirect products. When it is ready, then finding its atoms is easy (but not fast). Furthermore, the program can print the atoms and then the user can easily see that (in most cases) they are atoms even in the full direct product, so they are surely atoms in the subdirect product. So, it is only the generation of the subdirect product that mostly needs a real verification. When the program computes the double spectrum, the input file cannot rely on  \eqref{eqdHtrsbmDnMsb}. Then, as it is easy to see and it is pointed out in Berman and Wolk~\cite{bermanwolk}, the subdirect product of all assignments is the free lattice on three generators in the  variety we are dealing with.
 We needed to compute $\DS(\mvar L_1)$, $\DS(\mvar L_3)$, and $\DS(\mvar L_4)$ in the proof of Theorem~\ref{thmmain}\eqref{thmmainb}, and only the assignment for $M_3$ has to be removed from their input files to compute the corresponding free lattices. 
We did so and the program reported that the free lattices on three generators in the varieties $\mvar L_1$, $\mvar L_3$, and $\mvar L_4$ consist of
$178$, $2\,811$, and $821$ elements, respectively. The \emph{same numbers} have been given by Berman and Wolk~\cite[page 274]{bermanwolk}. 
Besides that the above-mentioned coincidence increases our trust in the program, it also verifies three of the input files. The rest of these files can be checked manually since they are readable text files and they are available from the author's website. 
\nothing{Furthermore, both the source file of the program and its input files are attached the arXiv version of the paper as appendices.}\nothing{xxx}
\end{remark}

\goodbreak

\section{How far can we go?}\label{sect:howfar}
We can only give a modest partial answer to the question above. We begin with an example.

\begin{example}[stories about the lattice $U_8$]\label{exmpl:18} 
While the largest number of atoms of a three-generated lattice in the scope of the previous section is at most $6$, we know practically nothing on this number in case of other lattices. This is why the subdirectly irreducible lattice $U_8$ in the middle of Figure~\ref{figaa} and the variety $\var U_8:=\HSP \set{U_8}$ are of some interest. With the list
\begin{verbatim}
\lattice=U8 \with x=a y=b z=c
\lattice=U8 \with x=a y=c z=b
\lattice=U8 \with x=b y=a z=c
\lattice=U8 \with x=b y=c z=a
\lattice=U8 \with x=c y=a z=b
\lattice=U8 \with x=c y=b z=a
\end{verbatim}
of assignments, quoted from the corresponding input file, the program computed the six-fold subdirect product of $U_8$. This subdirect product consists of $47\,092$ elements and it has 18 atoms; the computation took 27 minutes. After observing that $L_4\in \var U_8$ and adding three more assignments,
\begin{verbatim}
\lattice=L4 \with x=a y=b z=c
\lattice=L4 \with x=b y=c z=a
\lattice=L4 \with x=c y=b z=a
\end{verbatim}
to the first six, the nine-fold subdirect product consists of $61\,608$ elements but the number of atoms is still 18; the computation took  56 minutes.
Finally, it took six hours and fourteen minutes to compute the subdirect product for all subsets of these nine assignments. The information we obtained in this way is that
\begin{equation}
\set{ 1, 2, 3, 4, 5, 6, 8, 9, 12, 15, 18}\subseteq \AS(\var U_8)\subseteq \AS(\mvar U_8),
\label{eqdzhGrhcstrNdc}
\end{equation}
where $\mvar U_8=\var U_8\vee\mvar N_5$, as usual. It is straightforward to see that for every lattice $K$,
\begin{equation}
\text{if $K\in\HSs\set{U_8}$ and $|K|<|U_8|$, then $K\in\var L_3\vee \var L_4$.}
\label{eqtxtHzrGmgRmmlr}
\end{equation}
Hence, it follows easily from B. J\'onsson's results, see \eqref{pbxBJg} and \eqref{pbxBJktt}, that 
\begin{equation}
\text{$\mvar U_8$ and $\var U_8$ cover  $\mvar L_3\vee \mvar L_4\in\bhhrom$ and  $\var L_3\vee \var L_4$,}
\label{eqtxtZfhfDnKr}
\end{equation}
respectively, in the lattice $\alllat$ of all lattice varieties. 
We do not know if the first inclusion in \eqref{eqdzhGrhcstrNdc} is proper or not, and we only guess that the second one might be an equality. Together with the covering  $\mvar L_3\vee \mvar L_4\prec \mvar U_8$, this shows that  even a little step out of $\bhhrom$ can bring lots of changes and difficulties.

The subdirectly irreducible lattices $\chain 2$, $N_5$, and $L_3$
also belong to $\var U_8$. If we added  the corresponding assignments to the nine mentioned above and the possible constraints then, by a rough estimation, it would take  months or, rather, years  to compute $\AS(\var U_8)$. This is why  neither $\AS(\var U_8)$, nor $\AS(\mvar U_8)$ is  given in this paper. 

As opposed to $\AS(\mvar U_8)$, which we do not know, the results proved in this paper enable us to determine $\CS({\mvar U_8})$ as follows. 
Applying Lemma~\ref{keylemma} first with $(\var M, L_3)$ playing the role of $(\var V,K)$ and then $(\mvar L_3,L_5)$ playing the same role, we obtain that the free lattice on three generators in $\mvar L_3\vee \mvar L_5$ is finite. 
By \eqref{eqzMywRdlWtzls} and Theorem~\ref{thmmain}\eqref{thmmaina}, 
\begin{equation}
\AS(\mvar L_3)=\AS(\mvar L_3\vee \mvar L_5)=\set{1,2,3,4}. 
\label{eqdssjzRgrW}
\end{equation}
(The first equality above is only for a later reference.)
Since $L_3$ is a selfdual lattice and $L_5=\dual{L_4}$, the dual of \eqref{eqtxtHzrGmgRmmlr}
gives that all proper homomorphic images and all proper sublattices of $\dual{U_8}$ belong to $\mvar L_3\vee \mvar L_5$. Applying Lemma~\ref{keylemma}\eqref{keylemmaa} to the above-mentioned facts, we obtain that
$\AS(\dual{\mvar U_8})=\set{1,2,3,4}$. Therefore,
\begin{equation}
\CS({\mvar U_8})=\set{1,2,3,4}.
\label{eqhRtRjFtbS}
\end{equation}
\end{example}
 
\begin{remark} Assume that $\var X$ is a lattice variety such that $\mvar X\in\bhhrom$; see \eqref{pbxmskSgnKa} for the notation. No systematic study of the relationship between the spectra of $\var X$ and those of $\mvar X$ is targeted in the present paper; we only mention the following.
As opposed to what the results and examples formulated so far suggest, $\AS(\var X)$ and $\AS(\mvar X)$ can be distinct. For example, compare  \eqref{eqdssjzRgrW} with
\begin{equation}
\AS(\var L_3)=\AS(\var L_3\vee \var L_5)=\set{1,2,3},
\label{eqdsizRsprhRmhgrD}
\end{equation}
which was computed by our computer program in less than a second.
\end{remark}

\begin{remark} Similarly to the argument showing \eqref{eqhRtRjFtbS}, 
the following counterpart of \eqref{eqhRtRjFtbS} follows from  \eqref{eqdsizRsprhRmhgrD}, the dual of \eqref{eqtxtHzrGmgRmmlr}, 
and Lemma~\ref{keylemma}\eqref{keylemmaa}:
\begin{equation}
\CS({\var U_8})=  \AS(\HSP{\set{\dual{U_8}}}) =  \set{1,2,3}.
\label{eqhRgrZHsTzGr}
\end{equation}
\end{remark}

Theorem~\ref{thmmain} shows that for many nontrivial intervals $[\var V,\var W]$ of $\bhhrom$,  it may happen that  $\AS(\var V)=\AS(\var W)$. 
The length of such an interval $[\var V,\var W]$ is at most 22 (and it is  22 for  $[\var V,\var W]=[\mvar L_4, \var T_{23}]$). Using the method of the proof of Theorem~\ref{thmmain}, one can derive the following  consequence of (the Key) Lemma~\ref{keylemma} in a straightforward way.

\begin{remark}\label{remarkHqrhhqxG} If $\var V$ is a lattice variety such that $\afree{\var V}3$ is finite and $k$ is a positive integer, then there exists a lattice variety 
$\var W$ such that $\var V\leq \var W$ in $\alllat$, the interval $[\var V,\var W]$ is of length $k$, and $\AS(\var W)=\AS(\var V)$. 
\end{remark}

No survey of the triples $(\var V,\var W,k)$ with components taken from Remark~\ref{remarkHqrhhqxG} is targeted in the present paper, and  it is not clear whether a \emph{satisfactory} survey would be possible. We only mention that (the Key) Lemma~\ref{keylemma} is not sufficient to yield 
every possible triples. 
For example, if $\var V= \mvar N_5$ and $\var V<\var W\in [\mvar N_5,\mvar H_3]$, then Lemma~\ref{keylemma} is not applicable but we know from 
Theorem~\ref{thmmain} that $\AS(\var W)=\AS(\var V)$.
In addition to what (the Key) Lemma~\ref{keylemma} offers, there is another straightforward way to find an extension $\var W$ of $\var V$, not always a proper extension, such that $\AS(\var W)=\AS(\var V)$ simply because the class of three-generated lattices remains  unchanged. Namely, we can apply the following observation for lattices with $k=3$.

\begin{observation}\label{ObservLG} Let $\var V$ be a variety algebras and let $k$ be a positive integer. Then there exists a largest variety $\var W$ such that $\var V\subseteq \var W$ but every algebra of $\var W$ that can be generated by at most $k$ elements belongs to $\var V$.  
\end{observation}

To give an example,  if $V$ is the lattice variety $\var M_3=\HSP\set{M_3}$ and $k=3$, then a straightforward (but omitted) argument shows that $\var W$ above is the variety $\var M$ of modular lattices; in this case the interval $[\var V,\var W]$ is not of finite length and its cardinality is continuum.

\begin{proof}[Proof of Observation~\ref{ObservLG}] Based on Birkhoff's classical characterization of equational classes as varieties, see Birkhoff~\cite{birkhoffHSP}, the proof is almost trivial. For a set $\Sigma$ of identities, let $\Sigma(k)$ denote the set of at most $k$-variable identities belonging to $\Sigma$.  For a variety $\var X$, let $\Sigma_{\var X}$ be the set of all identities that hold in $\var X$. Let $\boldsymbol G$ be the set of all varieties $\var U$ such that $\var U$ is of the same signature as $\var V$, $\var V\subseteq \var U$, and $\Sigma_{\var U}(k)=\Sigma_{\var V}(k)$. Then
$\Gamma:=\bigcap\set{ \Sigma_{\var U}: \var U\in  \boldsymbol G}$
is closed with respect to the inference rules since so are all the $\Sigma_{\var U}$,  $\var U\in  \boldsymbol G$. Hence, by Birkhoff's theorem, the models of $\Gamma$ form a variety $\var W$ and $\Sigma_{\var W}=\Gamma$. Since $\Gamma(k)=\bigcap\set{\Sigma_{\var U}(k): \var U\in \boldsymbol G}=\set{\Sigma_{\var V}(k): \var U\in \boldsymbol G}
=\Sigma_{\var V}(k)$, we have that $\Sigma_{\var W}(k)=\Gamma(k)=\Sigma_{\var V}(k)$. Hence, 
$\var W\in \boldsymbol G$. Since $\Gamma=\Sigma_{\var W}$ is the smallest element of $\set{\Sigma_{\var U}: \var U\in \boldsymbol G}$, we obtain that $\var W$ is the largest member of $\boldsymbol G$, as required. 
\end{proof}

In order to prove part \eqref{thmmaina}, it suffices to prove that 
\begin{equation}\left\{\,\,
\parbox{8cm}{the atom spectra of the eight members of 
the yellow interval $[\mvar N_5, \mvar H_6]$ that are indicated by vertices in Figure~\ref{figdd} are the same as stated in the theorem.}\right.
\label{pbxbNlCsKrldFzZGh}
\end{equation}
Indeed, if \eqref{pbxbNlCsKrldFzZGh} held, then \eqref{eqknXmpHle} and \eqref{eqVlzrbTssGma} would yield the validity of part \eqref{thmmaina} for all the 32 varieties belonging to the yellow interval $[\mvar N_5, \mvar H_6]$. Thus, 
it would follow from
\eqref{eq:zlnkgchmZbrJng} and 
 and Proposition~\ref{propgMncmp}\eqref{propgMncmpe} 
that any two varieties belonging to the same 
layer have the same atom spectrum, whereby \eqref{pbxbNlCsKrldFzZGh} would  imply the validity of part \eqref{thmmaina}.

\leftline{\hfill \includegraphics[width=\textwidth]{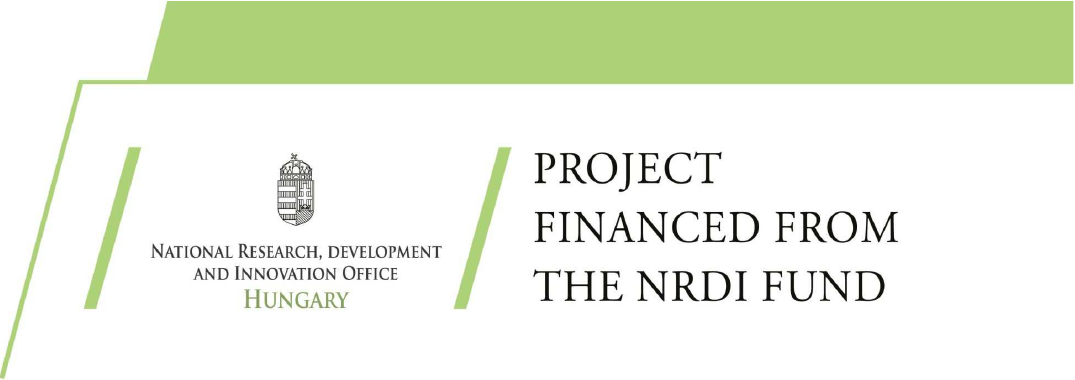}}

\end{document}